\documentclass[reqno,12pts]{amsart}
\usepackage{amssymb,amsmath,amsthm,amstext,amsfonts}
\usepackage{amsmath,amstext,amsthm,amsfonts}
\usepackage[dvips]{graphicx}
\usepackage[dvips]{graphicx}
\usepackage{color}

\makeatletter \@addtoreset{equation}{section} \makeatother

\renewcommand\thetable{\thesection.\@arabic\c@table}

\theoremstyle{plain}
\newtheorem{maintheorem}{Theorem}

\newtheorem{maincorollary}{Corollary}

\newtheorem{theorem}{Theorem }[section]
\newtheorem{proposition}[theorem]{Proposition}
\newtheorem{lemma}[theorem]{Lemma}

\theoremstyle{definition} \theoremstyle{remark}
\newtheorem{remark}[theorem]{Remark}

\newtheorem{definition}[theorem]{Definition}

\newcommand{\al} {\alpha}

\newcommand{\vep}{\varepsilon}

\newcommand{\si} {\sigma}       \newcommand{\Si}{\Sigma}

\newcommand{\N}{\mathbb{N}}
\newcommand{\R}{\mathbb{R}}

\newcommand{\topp}{\operatorname{top}}

\newcommand{\Ptop}{P_{\topp}}

\newcommand{\cE}{\mathcal{E}}
\newcommand{\cM}{\mathcal{M}}

\newcommand{\cF}{\mathcal{F}}

\def\ds{\displaystyle}

\begin{document}

\title{Multifractal analysis of the irregular set for almost-additive sequences via large deviations}

\author{ Thiago Bomfim and Paulo Varandas}

\address{Thiago Bomfim, Departamento de Matem\'atica, Universidade Federal da Bahia\\
Av. Ademar de Barros s/n, 40170-110 Salvador, Brazil.}
\email{tbnunes@ufba.br}
\urladdr{https://sites.google.com/site/homepageofthiagobomfim/}

\address{Paulo Varandas, Departamento de Matem\'atica, Universidade Federal da Bahia\\
Av. Ademar de Barros s/n, 40170-110 Salvador, Brazil  \\ \&
Faculdade de Ci\^encias da Universidade do Porto, Rua do Campo Alegre 1021/1055, 4169--007 Porto,
Portugal.}
\email{paulo.varandas@ufba.br }
\urladdr{http://www.pgmat.ufba.br/varandas}

\date{\today}

\begin{abstract}
In this paper we introduce a notion of free energy and large deviations rate function for asymptotically
additive sequences of potentials via an approximation method by families of continuous potentials.
We provide estimates for the topological pressure of the set of points whose non-additive sequences
are far from the limit described through Kingman's sub-additive ergodic theorem and give some applications
in the context of Lyapunov exponents for diffeomorphisms and cocycles, and Shannon-McMillan-Breiman
theorem for Gibbs measures.
\end{abstract}


\maketitle

\section{Introduction}

The study of the {\color{black}thermodynamic} formalism for maps with some hyperbolicity has drawn the attention
of many researchers from the theoretical physics and mathematics communities in the last decades. A particular topic of
interest in ergodic theory is to obtain limit theorems, the characterization of level sets, the velocity of convergence and to characterize
the set of points that do not converge, often called the irregular set. The general concept of multifractal analysis is to decompose
the phase space
in subsets of points which have a similar dynamical behavior and to describe the size of each of such subsets from the geometrical or
topological viewpoint. We refer the reader to the introduction of \cite{Olsen} and references therein for an excellent historical
account.
The study of the topological pressure or Hausdorff dimension of the level and the irregular sets can be
traced back to Besicovitch. Such a multifractal analysis program has been carried out successfully to deal with
self-similar measures~\cite{Olsen,LiWu,OlsenWinter},
Birkhoff averages~\cite{PW97,PW01, Daniel, special, CZ13, BV14},
Lyapunov spectrum~\cite{DK, BPS97, BG06, Ban,GR, Shu, Feng, T},
and in the case of simultaneous level sets~\cite{Climenhaga, Cli13},
and local entropies~\cite{TV99} just to quote some directions and contributions.
For additive sequences, level sets carry all ergodic information. In fact, by Birkhoff's ergodic theorem all ergodic measures give full weight to some level set. On the other hand, the irregular set may have full Hausdorff dimension or full topological pressure meaning that {\color{black}it can certainly not be omitted from} the topological or geometrical point of view (see e.g. \cite{Daniel}).
In particular the irregular set associated to Birkhoff sums for maps with some hyperbolicity has a rich multifractal
structure (see e.g.\cite{BV14}).
Due to the more recent developments of non-additive thermodynamic formalism, including~\cite{Barreira,BD09,Ba10,
Feng1,Feng2,IY, Mummert}, it is natural to ask whether there can exist a unified approach for the multifractal analysis
for certain classes of non-additive sequences of observables.

Here we aim to provide a multifractal analysis of the irregular set in the non-additive setting that we
now describe.
Fix $M$ a compact metric space and $f : M \rightarrow M$ a continuous dynamical system.
A sequence
$\Phi=\{\varphi_n\}\subset C(M, \R)^{\mathbb N} $  is a \emph{sub-additive}
 sequence of potentials if
$\varphi_{m+n}\le \varphi_{m}+\varphi_{n}\circ f^m$ for every
$m,n\ge 1$.
We say that the sequence  $\Phi=\{\varphi_n\}\subset C(M, \R)^{\mathbb N} $ is an
\emph{almost additive} sequence of potentials, if there exists a
uniform constant $C>0$ such that $\varphi_{m}+\varphi_{n}\circ
f^m-C \le \varphi_{m+n}\le \varphi_{m}+\varphi_{n}\circ f^m+C$ for
every $m,n\ge 1$. Finally, we say that $\Phi=\{\varphi_n\}\subset
C(M, \R)^{\mathbb N} $ is an \emph{asymptotically additive} sequence of potentials,
if for any $\xi>0$ there exists a continuous function
$\varphi_\xi$ such that
\begin{equation}\label{eq.asymptotic}
\limsup_{n\to\infty} \frac1n \left\| \varphi_n - S_n \varphi_\xi
\right\|_{\infty} <\xi
\end{equation}
where $S_n\varphi_\xi=\sum_{j=0}^{n-1}\varphi_\xi\circ f^j$
denotes the usual Birkhoff sum, and $||\cdot||_{\infty}$ is the sup norm on
the Banach space $C(M, \R)$.
It follows from the definition that if
if $\Phi=\{\varphi_n\}$ is almost additive then  there exists $C>0$ such that the sequence
$\Phi_C=\{\varphi_n+C\}$ is sub-additive.
Morever, if $\Phi=\{\varphi_n\}$ is almost additive then it is asymptotically additive (see e.g. \cite{ZZC11}).
By Kingman's subadditive ergodic theorem it follows that for every
sub-additive sequence $\Phi=\{\varphi_n\}$ and every
$f$-invariant ergodic probability measure $\mu$ so that $\varphi_1\in L^1(\mu)$ it holds
\begin{equation}\label{eq.Kingman}
\lim_{n\to\infty} \frac1n \varphi_n(x)
    =\inf_{n\geq 1} \frac1n \int \varphi_n \;d\mu=: \cF_*(\Phi,\mu),
    \quad \text{ for $\mu$-a.e. $x$}.
\end{equation}
The study of the multifractal spectrum associated to non-additive sequences of potentials arises naturally
in the study of Lyapunov exponents for non-conformal dynamical systems.
Feng and Huang~\cite{Feng} used subdiferentials
of pressure functions to characterize the topological pressure of the level sets
$$
\Big\{ x\in M : \lim_{n\to\infty} \frac1n \psi_n(x)=\alpha \Big\}
$$
for asymptotically sub-additive and asymptotically additive families $\Psi=\{\psi_n\}_n$.
Zhao, Zhang and Cao \cite{ZZC11} proved that
\textcolor{black}{if $f$ satisfies the specification property and $\Psi$ is any asymptotically additive sequence
of continuous potentials then
either the irregular set the $X(\{\psi_{n}\})$ (which consists of the points $x \in M$ such that the limit of
$\frac{1}{n}\psi_{n}(x)$ does not exist) is empty or carries full topological pressure for $f$ with respect to all asymptotically additive potential. This result proves that the irregular set often exhibits full topological complexity and provides the
starting point for a finer multifractal analysis description of  the irregular set that we address in this paper}.
We will be most interested in the analysis of the sets
$$
\overline{X}_{\mu,\Psi , c}
    := \Big\{ x \in M : \limsup_{n\to\infty} \Big| \frac{1}{n}\psi_{n}(x) - \mathcal{F}_{\ast}(\mu , \Psi) \Big| \geq c \Big\}
$$
and
$$
\underline{X}_{\mu,\Psi , c}
    := \Big\{ x \in M : \liminf_{n\to\infty} \Big| \frac{1}{n}\psi_{n}(x) - \mathcal{F}_{\ast}(\mu , \Psi) \Big| \geq c \Big\},
$$
where  $\Psi = \{\psi_{n}\}$ is an asymptotically additive or sub-additve sequence of observables, $c > 0$
and $\mu$ is an equilibrium state. More precisely, what are the properties and regularity of
the topological pressure functions $c\mapsto P_{\underline{X}_{\mu,\Psi , c}}(f,\Phi)$ and
$c\mapsto P_{\overline{X}_{\mu,\Psi , c}}(f,\Phi)$?
Such characterization and interesting applications for sequences $\Psi=\{\psi_n\}$ where $\psi_n=S_n\psi$
are Birkhoff sums were obtained in \cite{BV14}.

One of our purposes here is to characterize the sets $\overline{X}_{\mu,\Psi , c}$ and
$\underline{X}_{\mu,\Psi , c}$ thus extending the results from \cite{BV14} for almost additive sequences
of potentials, in which case a {\color{black}thermodynamic} formalism is available (see e.g \cite{Barreira, Mummert, BD09, Ba10}). One motivation is the study of Lyapunov exponents since beyond the one-dimensional and
conformal setting the situation is much less understood.

The strategy in this paper is to approximate averages of almost additive sequences by genuine Birkhoff averages of continuous functions. In the setting of subshifts of finite type we prove that almost additive sequence $\Psi$ are asymptotically additive
and that the sequences $\frac{\psi_n}n$ are uniformly approximated by Birkhoff means of sequences of potentials
can be chosen to have further regularity (c.f. Proposition~\ref{rmkasymp}) which,
in the case of uniformly expanding dynamics, we choose to be H\"older continuous. The key step is to prove
that the thermodynamical limiting objects that, we will detail below, do not depend on the approximating family.

We introduce a free energy function $\mathcal{E}_{f , \Phi,\Psi} (\cdot)$ and a rate function
$I_{f , \Phi,\Psi} (\cdot)$ obtained as limit of Legendre transforms that does not depend on the
family of approximations chosen and it is strictly convex in a neighborhood of $\cF_*(\Psi,\mu_\Phi)$
if and only if $\Psi$ is not cohomologous to a constant.
 This characterization using
the Legendre transform and the variational formulation for the large deviations rate function is enough to
obtain a functional analytic expression for the large deviations rate function obtained in \cite{Va13},
opening the way to study its continuous and differentiable dependence.
In the case of repellers, when the irregular set $X(\{\psi_n\})$ is nonempty then it carries full topological pressure.
We prove that $P_{\underline{X}_{\mu,\Psi , c}}(f,\Phi) \le P_{\overline{X}_{\mu,\Psi , c}}(f,\Phi)
<P_{\text{top}}(f,\Phi)$  for any positive $c>0$ meaning that the set
$X(\{\psi_n\})\cap \overline{X}_{\mu,\Psi , c}$ does not have full pressure. This means that irregular points
responsable for the topological pressure are those whose values are arbitrarily close to the mean.
In fact, in the case that $\Phi=0$ and $\mu_o$ denotes the maximal entropy measure
we give precise a characterization of the topological entropy of these sets in terms of the large
deviations rate function and deduce that   $\R^+_0 \ni c\mapsto h_{\underline{X}_{\mu_{0},\Psi,c}}
(f)=h_{\overline{X}_{\mu_{0},\Psi,c}}(f)$ is continuous, strictly decreasing and concave in a neighborhood of zero. (we refer to Section~\ref{Statement of the main results} for precise statements).

This paper is organized as follows. In Section~\ref{Statement of the main results} we introduce
the necessary definitions and notations and state our main results. Section~\ref{sec:free}
is devoted to the definition of these generalized notions of free energy and Legendre transforms
and to the proof of Theorem~\ref{legtrans}. Section~\ref{sec:frac} is devoted to the proof of the multifractal
analysis of irregular sets. Finally in Section~\ref{Examples}  we provide some examples and applications
of our results in the study of Lyapunov exponents for linear cocycles, non-conformal repellers and sequences
arising from Shannon-McMillan-Breiman theorem for entropy.

\section{Statement of the main results}\label{Statement of the main results}

This section is devoted to the statement of the main results. Our first results concern the regularity of the
pressure function and the Legendre transform of the free energy function and its consequences to large
deviations.

\subsubsection*{Topological pressure and equilibrium states}

Given an asymptotically additive sequence of potentials $\Phi = \{\phi_{n}\}$ and a arbitrary invariant
set $Z \subset M$ it can be defined the topological pressure $P_{Z}(f,\Phi)$ of $Z$ with respect to $f$ and $\Phi$
by means of a \textcolor{black}{Carath\'eodory} structure.
Let us mention that in the case that $\Phi=\{\phi_n\}$ with $\phi_n=S_n\phi$ for some continuous potential $\phi$
then $P_{Z}(f,\Phi)$ is exactly the usual notion of relative topological pressure for $f$ and $\phi$  on $Z$
introduced by Pesin and Pitskel. We refer the reader to \cite{pesin} for a complete account on \textcolor{black}{Carath\'eodory} structures.
Alternativelly, for asymptotically additive sequence of potentials the topological pressure can be
defined using the variational principle proved in \cite{Feng}
$$
P_{\topp}(f,\Phi) = \sup  \{ h_{\mu}(f) + \mathcal{F}_{\ast}(\Phi,\mu) : \mu \;\mbox{is an } \textcolor{black}{f\mbox{-invariant probability,}}\; \mathcal{F}_{\ast}(\Phi,\mu) \neq -\infty \}
$$
(see Subsection~\ref{sec:pressure} for more details.)
If an invariant probability measure $\mu_\Phi$ attains the supremum then we say that it is an \emph{equilibrium state} for $f$ with respect to $\Phi$. In this sense, equilibrium states are invariant measures
that reflect the topological complexity of the dynamical system. In many cases equilibrium states arise as (weak) Gibbs measures.
Given a sequence of functions $\Phi = \{\phi_{n}\}$, we say that a probability $\mu$ is a \emph{weak Gibbs measure} with respect to $\Phi$ on $\Lambda \subset M$ if there exists $\vep_0>0$ such that for every
$0<\vep<\vep_0$ there exists a positive sequence $(K_{n}(\vep))_{n \in \N}$ so that
$\lim \frac{1}{n}\log K_{n}(\vep) = 0$ such that for every $n\ge 1$ and  $\mu$-a.e. $x \in \Lambda$
$$
K_{n}(\vep)^{-1} \leq \frac{\mu(B(x,n,\vep))}{e^{-n P + \phi_{n}(x)}} \leq K_{n}(\vep).
$$
If, in addition, $K_{n}(\vep) = K(\vep)$ does not depend of $n$ we will say that $\mu$ is a \emph{Gibbs measure}.
Gibbs measures arise naturally in the context of hyperbolic dynamics: given a basic set $\Omega$ for
a diffeomorphism $f$ Axiom A (or $\Omega$ repeller to $f$) it is known that every almost additive potential $\Phi$ satisfying
\begin{equation}\label{equa5}
(\mbox{bounded distortion}) \; \exists A, \delta > 0 : \sup_{n \in \N}\gamma_{n}(\Phi,\delta) \leq A,
\end{equation}
where $\gamma_{n}(\Phi,\delta) := \sup\{|\phi_{n}(y) - \phi_{n}(z)| : y,z \in B(x,n,\delta) \},$ admits a unique equilibrium state $\mu_{\Phi}$ is a Gibbs measure with respect to $\Phi$ on $\Omega$
(see \cite{Barreira} and \cite{Mummert} for the proof). This concept in the additive context was introduced by
Bowen \cite{Bowen} to prove uniqueness of equilibrium states for expansive maps with the specification
property and it is weaker than the bounded distortion condition introduced by Walters
(cf. \cite{WaltersBD}). We will define now a weaker bounded distortion condition: we will say that a
sequence of continuous functions $\Phi = \{\varphi_{n}\}$ satisfies the \emph{weak Bowen condition} if
{\color{black}
\begin{equation}\label{eweakBowen}
\exists \delta > 0 \colon \lim_{n \to +\infty}\frac{\gamma_{n}(\Phi,\delta)}{n} = 0.
\end{equation}
In \cite{Va13}, Zhao and the second author obtained large deviations results for weak Gibbs measures and sub-additive observables with the weak Bowen condition.}
\textcolor{black}{We say the sequence $\Phi = \{\phi_{n}\}$  satisfies the \emph{tempered distortion condition} if
\begin{equation}\label{eq:tempered}
\lim_{\epsilon \to 0}\lim_{n \to +\infty}\frac{\gamma_{n}(\Psi , \epsilon)}{n} = 0.
\end{equation}
It is immediate from the definition that condition~\eqref{eq:tempered} is weaker than ~\eqref{eweakBowen}.
}

\subsubsection*{Legendre transforms in the non-additive case}

In this section we will assume that $M$ is a Riemannian manifold, $f : M \rightarrow M$ is a $C^{1}$ map,
and $\Lambda \subset M$ is a isolated repeller such that $f\mid_\Lambda$ is topologically mixing.
{\color{black} Although we will restrict to the context of repellers for simplicity the results
on the thermodynamic formalism needed here also hold
for subshifts of finite type and, for that reason,
our results also are valid for subshifts of finite type.}
For any almost additive potential $\Phi$ satisfying the bounded distortion condition we know by \cite{Barreira} that there is a unique equilibrium state for $f$ with respect to $\Phi$, and we denote it by $\mu_{\Phi}$.
Later, Barreira proved also the differentiability of the pressure function.

\begin{proposition}\cite[Theorem 6.3]{Ba10} \label{nonbarr}
Let $f$ be a continuous map on a compact metric space and assume that $\mu \mapsto h_\mu(f)$
is upper semicontinuous. Assume that $\Phi$ and $\Psi$ are almost additive sequences satisfying the bounded distortion condition and that there exists a unique equilibrium state for the family
$\Phi +t \Psi$ for every $t\in \mathbb R$.
Then the function $\R \ni t \mapsto P_{\topp}(f, \Phi + t\Psi)$ is $C^{1}$
 and \textcolor{black}{$\frac{d}{dt} \Ptop(f, \Phi + t \Psi) = \mathcal{F}_{\ast}(\Psi , \mu_{\Phi + t\Psi})$}.
\end{proposition}

For any almost additive sequences of potentials $\Phi$ and $\Psi$ we define the \emph{free energy function}
associated to $\Phi$  and $\Psi$ by
\begin{equation}\label{eq:defEt}
\mathcal{E}_{f , \Phi,\Psi} (t): = P_{\topp}(f , \Phi + t\Psi) - P_{\topp}(f , \Phi).
\end{equation}
for $t \in \R$ such that the right hand side is well defined.

\textcolor{black}{
\begin{remark}
The previous definition is motivated by the following fact that for the additive setting: given observables $\phi,\psi$ with
bounded distortion the free energy function $\cE_{f,\phi,\psi} $, defined originally by
\begin{equation*}
\cE_{f,\phi,\psi}(t)
	= \limsup_{n\to\infty} \frac1n \log \int e^{t S_n  \psi} \,d\mu_{f,\phi},
\end{equation*}
where $S_n  \psi=\sum_{j=0}^{n-1} \psi\circ f^j$ is the usual Birkhoff sum, can often be proved to
satisfy $\mathcal{E}_{f , \phi,\psi} (t) = P_{\topp}(f , \phi + t\psi) - P_{\topp}(f , \phi)$ (see e.g.~\cite{DK,BCV13}).
\end{remark}
}

\textcolor{black}{
If $ \Phi $ and $ \Psi $ are almost additive sequences satisfying the bounded distortion condition then there
exists a unique equilibrium state for the family $\Phi+t\Psi$ for every $t\in\mathbb R$. Then
the differentiability of the pressure function follows from Proposition~\ref{nonbarr}  and, consequently,
the free energy function $t \mapsto \mathcal{E}_{f , \Phi,\Psi}(t)$ is $C^{1}$.}

\begin{proposition}\label{rmkasymp}
Let $H$ be a dense subset of the continuous functions $C(M,\mathbb R)$ in the usual sup-norm $\|\cdot\|_\infty$. If $\Psi=\{\psi_n\}$ is an asymptotically additive sequence of observables
then there exists $(0 ,1) \ni \vep \mapsto g_{\vep} \in H$ so that for any $\vep>0$
$$
\limsup_{n \to +\infty} \frac{1}{n}||\psi_{n} - S_{n}g_{\vep}||_{\infty} < \vep.
$$
\end{proposition}

\begin{proof}
Since $\Psi=\{\psi_n\}$ is an asymptotically additive sequence of observables
there exists a family $(\tilde g_{\vep})_\vep$ of continuous functions  such that for every small $\vep > 0$
we have that $\limsup_{n \to +\infty}\frac{1}{n}||\psi_{n} - S_{n} \tilde g_{\vep}||_{\infty} < \vep/2$. Since
$H \subset C(M,\mathbb R)$ is dense then there exists a family $( g_{\vep})_\vep$ of observables in $H$
such that $\| g_\vep -\tilde g_\vep\|_\infty < \vep/2$ for all $\vep$. The later implies that the Birkhoff averages
are $\vep/2$ close, thus proving the lemma.
\end{proof}

Since the {\color{black}thermodynamic} formalism for expanding maps is well adapted the space
of H\"older continuous potentials we will take $H=C^{\alpha}(M , \R)$ for some
$\al\in(0,1)$.   Given $\Psi =\{\psi_{n}\}_{n}$ almost additive it follows e.g.  from \cite[Proposition~2.1]{ZZC11}
that this sequence is asymptotically additive and thus we can assume the approximations above are by
H\"older continuous functions. We will refer to such families of functions as an \emph{admissible family for $\Psi$}
and denote it by  $\{g_{\vep}\}_{\vep}$. In what follows let $\al\in(0,1)$ be fixed.

\begin{definition}
Let $\Psi$ be an almost additive sequence of observables. We will say that $\Psi$ is \emph{cohomologous to a constant} if there exists an admissible family $\{g_{\vep}\}_{\vep}$
for $\Psi$ such that $g_{\vep}$ is cohomologous to a constant for every small $\vep \in (0 , 1)$,
that is, there exists a constant $c_\vep$ and a \emph{continuous} function $u_\vep$ so
that  $g_\vep = u_\vep\circ f -u_\vep+c_\vep$.
\end{definition}

A natural question is to understand which families are cohomologous to a constant. Such characterization
is assured by the next lemma.

\begin{lemma}\label{noncoho}
$\Psi = \{\psi_{n}\}_{n}$ is cohomologous to a constant if
only if $(\frac{\psi_{n}}{n})_n$ is uniformly convergent to a constant.
\end{lemma}

\begin{proof}
On the one hand, if $\Psi$ is cohomologous to a constant then there exists an admissible family
$\{g_{\vep}\}_{\vep}$
for $\Psi$ such that $g_{\vep}$ is cohomologous to a constant for every small $\vep \in (0 , 1)$, that is,
there are constants $c_\vep\in \mathbb R$ and continuous functions $u_\vep$ such that
$g_\vep=u_\vep\circ f - u_\vep+c_\vep$ and, consequently, $S_n g_\vep=u_\vep\circ f^n - u_\vep+c_\vep n$
for every small $\vep$.
Using the convergence given by equation~\eqref{eq.asymptotic} it follows that for every small $\vep$
$$
\limsup_{n\to\infty} \Big\|  \frac{\psi_n}{n} - c_\vep \Big\|_\infty
    = \limsup_{n\to\infty} \frac1n \Big\| \psi_n - S_{n}g_{\vep} + u_\vep\circ f^n - u_\vep\Big\|_\infty <\vep,
$$
which proves that $c=\lim_{\vep\to 0} c_\vep$ does exist and that $(\frac{\psi_{n}}{n})_n$ is uniformly convergent to
the constant $c$.
On the other hand, if $(\frac{\psi_{n}}{n})_n$ is uniformly convergent to a constant $c$ then take $g_\vep$ constant
to $c$ and notice that since $S_n g_\vep =c n$ then clearly
$$
\limsup_{n\to\infty} \frac1n  \Big\|  \psi_n -S_n g_\vep \Big\|_\infty=0.
$$
This finishes the proof of the lemma.
\end{proof}

\begin{remark}
Let us notice that the notion of cohomology for families of observables is slightly different from the corresponding
one for a fixed observable. Indeed, for instance by the previous lemma the family $\Psi=\{\psi_n\}_n$ with $\psi_n=\sqrt{n} \,w$
is cohomologous to the constant 0 although the observable $w: M \to \R$ may be chosen to be not  cohomologous to a constant.
\end{remark}

Observe that it follows from the definition that if $ \Psi $ is not cohomologous to a constant then there is a
admissible family $ \{g_{\vep}\}_{\vep} $ for $ \Psi $ and a sequence $(\vep_k)_k$ converging
to zero such that $ g_{\vep_k} $ is not cohomologous  to a constant for every $k\ge 1$.
If this is the case, the family $\vep \mapsto \tilde g_\vep$ given by
$\tilde g_\vep=g_{\vep_k}$ for every  $\vep_{k}\le \vep < \vep_{k-1}$ is so that $\tilde g_\vep$
is not cohomologous to a constant for every small $\vep$ (notice that these ``step
functions'' could be chosen in many different ways).
We will say that such family $ \{\tilde g_{\vep}\}_{\vep} $ is \emph{not cohomologous} to a constant.
Then, for simplicity, given any $ \Psi $ is not cohomologous to a constant we shall
consider the approximations by admissible sequences $(g_\vep)_\vep$ such that $g_\vep$ is not
cohomologous to a constant  for any small $\vep$.

Assume $ \Phi, \Psi $ are almost additive sequences of potentials with the bounded distortion condition such that $ \Psi $ is not cohomologous to a constant and let $(\varphi_{\vep})_{\vep} $ and  $\{g_{\vep}\}_{\vep} $ be admissible families for $\Phi$ and $ \Psi $ respectivelly. Then the well defined free energy function
$t\mapsto \mathcal E_{f, \varphi_{\vep}, g_{\vep}}$ is strictly convex and so it makes sense to compute the
Legendre transform $ I_{f, \varphi_{\vep}, g_{\vep}}(t) $ for every small $\vep \in (0, 1) $ and $ t \in \R $.
Since each $ g_{\vep} $ is not cohomologous to a constant it is a classical result that the following variational property holds
\begin{equation}\label{eq:variational}
I_{f,\varphi_{\vep},g_{\vep}}(\mathcal{E}'_{f,\varphi_{\vep},g_{\vep}}(t)) = t \mathcal{E}'_{f,\varphi_{\vep},g_{\vep}}(t) - \mathcal{E}_{f,\varphi_{\vep},g_{\vep}}(t)
\end{equation}
 for every small $\vep \in (0 , 1)$ and $t \in \R$ (see e.g. \cite{DK,BCV13}). Using this variational property we prove in Section~\ref{sec:free} that it is possible define the \emph{Legendre transform} of the corresponding free energy functions of $ \Psi $
as
$$
I_{f , \Phi, \Psi}(s) := \lim_{\vep \to 0}I_{f,\varphi_{\vep},g_{\vep}}(s),
$$
for every $ s \in \big(\inf_{t \in \R}\mathcal{F}_{\ast}(\Psi , \mu_{\Phi + t \Psi})
    \;,\; \sup_{t \in \R}\mathcal{F}_{\ast}(\Psi , \mu_{\Phi + t \Psi}) \big) $,  since
this limit will not depend of the choices of families $\{\varphi_{\vep}\}_{\vep}$ and $\{g_{\vep}\}_{\vep}$.
\textcolor{black}{Recall that, in view of Proposition~\ref{nonbarr}, the free energy function $\mathcal{E}_{f , \Phi,\Psi} (\cdot)$
is $C^1$.
We establish some properties of the Legendre transform $I_{f , \Phi, \Psi}(\cdot)$ as follows.}

\begin{maintheorem}\label{legtrans}
Let $ M $ be a Riemannian manifold, $f: M \rightarrow M $ be a $C^{1}$-map and $ \Lambda \subset M $ be an
isolated repeller such that $f\mid_\Lambda$ is topologically mixing.
 Let $\Phi$ and $\Psi$ be almost additive sequences satisfying the
bounded distortion condition
and assume that $\Psi$ is not cohomologous to a constant.
  The following properties hold:
 \begin{itemize}
  \item[i.] the Legendre transform of $ \Psi $ satisfies the variational property
  $$
  I_{f,\Phi,\Psi}(\mathcal{E}'_{f,\Phi,\Psi}(t)) = t \mathcal{E}'_{f,\Phi,\Psi}(t) - \mathcal{E}_{f,\Phi,\Psi}(t),
  $$
  for every $t \in \R$;
  \item[ii.] $I_{f,\Phi,\Psi}(\cdot)$ is a non-negative convex function and
  $$
  \inf_{s \in (a , b)} I_{f,\Phi,\Psi}(s) = \min\{I_{f,\Phi,\Psi}(a) , I_{f,\Phi,\Psi}(b)\}
  $$
  for any interval $(a , b) \subset \R$ not containing $\mathcal{F}_{\ast}(\Psi, \mu_{\Phi} )$
\textcolor{black}{
    \item[iii.]
   $
  I_{f,\Phi,\Psi}(s)
    = \inf_{\eta \in \mathcal{M}_{1}(f)}
    \{P_{\topp}(f, \Phi) - h_{\eta}(f) - \mathcal{F}_{\ast}(\Phi,\eta) : \mathcal{F}_{\ast}(\Psi,\eta) = s\}
  $
  \item[iv.]
  $I_{f,\Phi,\Psi}(s) = 0$ if only if $s=\mathcal{F}_{\ast}(\Psi,\mu_{\Phi} )$; moreover $s \mapsto I_{f,\Phi,\Psi}(s)$
   is strictly convex in an open neighborhood of $\mathcal{F}_{\ast}(\Psi,\mu_{\Phi} )$.
   }
 \end{itemize}
\end{maintheorem}

\subsubsection*{Large deviations results}

The variational relation obtained in Theorem~\ref{legtrans} is of particular interest in the study of
large deviations. In \cite{Va13}, the first author and Zhao proved several large deviations results for
sub-additive and asymptotically additive sequences of potentials. In the case of expanding maps
and almost additive sequences of potentials Theorem~\ref{legtrans} leads to the following immediate consequence:

\begin{maincorollary}\label{cor:LDP}
Let $ M $ be a Riemannian manifold, $f: M \rightarrow M $ be a $C^{1}$-map and $ \Lambda \subset M $ be an
isolated repeller such that $f\mid_\Lambda$ is topologically mixing. Let
$\Phi=\{\varphi_n\}$ be an almost additive sequence of
potentials satisfying the bounded distortion condition and
$\mu_\Phi$ be the unique equilibrium state for $f\mid_{\Lambda}$ with
respect to $\Phi$.
If $\Psi=\{\psi_n\}$ is a family of almost additive potentials
satisfying the bounded distortion condition then it satisfies the
following large deviations principle: given $F \subset \R$ closed it holds
that
$$
\limsup_{n\to \infty} \frac 1 n \log \mu_\Phi \left(
\left\{x\in M: \frac 1 n \psi_n(x)\in F \right\}\right)
    \leq -\inf_{s\in F} I_{f, \Phi,\Psi}(s)
$$
and also for every open set $E\subset \R$
$$
\liminf_{n\to \infty} \frac 1 n \log \mu_\Phi \left(
\left\{x\in M: \frac 1 n \psi_n(x)\in E \right\}\right)
    \geq -\inf_{s\in E} I_{f, \Phi,\Psi}(s).
$$
\end{maincorollary}

\begin{remark}
Although these quantitative estimates can be expected to hold for more general asymptotically
additive sequences,
one should mention that an extension of limit theorems from almost-additive to asymptotically additive
sequences of potentials is not immediate by any means. In fact, a simple example of an asymptotically additive
sequence of potentials can be written as
$\psi_n= S_n \psi +a_n$ depending on the sequence of real numbers $(a_n)_n$. If $\psi$ is H\"older
continuous and $a_n = o( \sqrt{n})$ then $(\psi_n)_n$ satisfies the central limit theorem. However, the CLT
fails in a simple way e.g. if $a_n = n^{\frac12 +\vep}$ for any $\vep>0$.
\end{remark}

\subsubsection*{Multifractal estimates for the irregular set}

Given an asymptotically additive sequence of observables $\Psi = \{\psi\}_{n}$ and $J \subset \R$ we denote
$$
\overline{X}_{J} =  \{x \in M : \limsup_{n \to +\infty}\frac{1}{n}\psi_{n}(x) \in J\}
$$
and
$$
\underline{X}_{J} =  \{x \in M : \liminf_{n \to +\infty}\frac{1}{n}\psi_{n}(x) \in J\}.
$$
\textcolor{black}{and let $X(J)$ denote the set of points $x \in \Lambda$ so that $\lim_{n \to +\infty} \frac1n \psi_{n}(x)
$ exists and belongs to $J$.}
For any $\delta > 0$ we denote by $J_{\delta}$ the $\delta-$neighborhood of the set $J$ and for a probability
measure $\mu$ we define
$$
L_{J,\mu} := -\limsup_{n \to +\infty}\frac{1}{n}\log \mu\Big(\{x \in M : \frac{1}{n}\psi_{n}(x) \in J\} \Big).
$$
We are now in a position to state our first main result concerning the multifractal analysis of the irregular set.

\begin{maintheorem}\label{thm:B}
Let $M$ be a compact metric space, $f: M\to M$ be continuous, $\Phi = \{\phi_{n}\}$ be an almost additive sequence of potentials with $P_{\topp}(f,\Phi) > -\infty$.
Assume that  $\mu_{\Phi}$ is the unique equilibrium state of $f$ with respect
the $\Phi$, that it is a weak Gibbs measure and that the sequence $\Psi = \{\psi_{n}\}$  satisfies at least one
of the following properties:
\begin{itemize}
\item[(a)] $\Psi$ is asymptotically additive, or
\item[(b)] $\Psi$ is a sub-additive sequence so that
        \begin{itemize}
        \item[i.] it satisfies the weak Bowen condition;
        \item[ii.] $\inf_{n \geq 1}\frac{\psi_{n}(x)}{n} > -\infty$ for every $x \in M$;
        \item[iii.] the sequence $\{\frac{\psi_{n}}{n}\}$ is equicontinuous.
        \end{itemize}
\end{itemize}
Then, for any closed interval $J \subset \R$ and any small $\delta > 0$,
$$
P_{\underline{X}_{J}}(f , \Phi)
    \leq P_{\overline{X}_{J}}(f , \Phi)
    \leq P_{\topp}(f , \Phi) - L_{J_{\delta},\mu_\Phi}
    \leq \Ptop(f , \Phi).
$$
\end{maintheorem}

{\color{black} We recall that under the assumptions of the theorem if, additionally, $f$ satisfies the specification property
and the irregular set is non-empty then it carries full topological pressure. The previous theorem shows that each of the sets
$\overline X_J$ and $\underline X_J$ do not carry full topological pressure provided $L_{J_{\delta},\mu_\Phi}>0$.}
In Remark~\ref{veryweak} we indicate few modifications which imply that the estimate
$P_{\underline{X}_{J}}(f , \Phi) \le P_{\topp}(f , \Phi) - L_{J_{\delta},\mu_\Phi} $ holds under the assumption
that $\mu_\Phi$ satisfies a \emph{pointwise weak Gibbs property}, namely, whenever there exists
$\vep_0>0$ such
that for every $0<\vep<\vep_0$ there exists
a positive sequence $(K_{n}(\vep))_{n \in \N}$ so that $\lim \frac{1}{n}\log K_{n}(\vep) = 0$ such that for
$\mu_\Phi$-a.e.
$x \in \Lambda$ there exists a subsequence $n_k(x)\to \infty$ (depending on $x$) satisfying
\begin{equation}\label{eq:veryweak}
K_{n_k(x)}(\vep)^{-1} \leq \frac{\mu_\Phi(B(x,n_{k}(x),\vep))}{e^{-n_{k}(x)P + \phi_{n_{k}(x)}}} \leq K_{n_k(x)}(\vep).
\end{equation}
From \cite[Theorem~B]{Va13} we know that if $\cF_*( \Psi,\mu_{\Phi}) \notin J_{\delta}$ then
$L_{J_{\delta},\mu_\Phi}>0$ and, consequently, the topological pressure of  both sets $\underline{X}_{J}$
and $\overline{X}_{J}$ is strictly smaller than  $\Ptop(f,\Phi)$. 
The bound $P_{\overline{X}_{\mu_{\Phi},\Psi , c}}(f , \Phi) \leq P_{\topp}(f , \Phi) - L_{J,\mu_\Phi}$
holds e.g. if $\delta \mapsto L_{J_{\delta},\mu_\Phi}$ is upper semicontinuous.
In the additive setting this question is overcomed by means of the functional analytic approach used to
define the Legendre transform of the free energy function.
{\color{black}Despite the fact that one}
 misses the functional analytic
approach our approximation method is still sufficient to obtain finer estimates in the uniformly hyperbolic
setting.

\begin{maincorollary}\label{nonadtreg}
{\color{black}
Let $d\ge 1$ and $f:\Sigma \to \Sigma$ be a topologically mixing one-sided subshift of finite type, where
$ \Sigma \subset \{1, \dots d\}^{\mathbb N}$.}
Assume $\Phi=0$ and $\Psi$ is almost additive
sequence of potentials satisfying the bounded distortion condition,
$\Psi$ is not cohomologous to a constant and $ \mathcal{F}_{\ast}( \Psi,\mu_{0}) = 0$,
where $ \mu_{0} $ is the unique maximal entropy measure for $ f $.
Then for any interval $J\subset \R$
$$
h_{\overline{X}_{J} }(f ) \leq h_{\topp}(f ) - I_{f,0,\Psi}(c_*),
$$
where $c_*$ belongs to the closure of ${J}$ is so that $\ds I_{f,0,\Psi}(c_*)=\inf_{s \in J} I_{f,0,\Psi}(s)$.
Moreover,
if $\overline{X}_{J} \neq \emptyset$ then $c_{\ast}$ is a point in the boundary of $J$ and
\begin{align*}
h_{\overline{X}_{J}}(f)
    & = h_{\underline{X}_{J}} (f) = h_{X(c_{\ast})} (f) = h_{X(J)}(f)
     = h_{\topp}(f) - I_{f,0,\Psi}(c_*),
\end{align*}
In particular $\R^+_0 \ni c\mapsto h_{\overline{X}_{\mu_{0},\Psi,c}}(f)$ is continuous, strictly decreasing
and concave in a neighborhood of zero.
\end{maincorollary}

Let us mention that the previous characterization of the topological entropy of level sets was available
in this setting due to Barreira and Doutor~\cite{BD09}, while we can expect analogous estimates
to hold for the topological pressure provided a generalization of the previous results to the context
of the topological pressure \textcolor{black}{holds.} \textcolor{black}{Moreover, the previous result holds for uniformly expanding repellers
with respect to some $C^1$-map on a compact manifold since these admit finite Markov partitions and can be
semi-conjugate to subshifts of finite type.}

\section{Free energy and Legendre transform}\label{sec:free}

\subsection{Non-additive topological pressure for invariant non-compact sets}\label{sec:pressure}

In this subsection we describe the notion of topological pressure for asymptotically additive potentials and not necessarily compact invariant sets.
Let $M$ be a compact metric space, $f :M \rightarrow M$ be a continuous map and $\Phi = \{\phi_{n}\}_{n}$ be
an asymptotically additive sequence of
continuous potentials.
The \emph{dynamical ball} of center $x \in M$, radius $\delta > 0,$ and length $n \geq 1$ is defined by
$$
B(x ,n, \delta) := \{y \in M : d(f^{j}(y) , f^{j}(x)) \leq \delta, \text{ for every } 0 \leq j \leq n \}.
$$
 Let $\Lambda \subset M$ be, fix $\vep > 0.$ Define $\mathcal{I}_{n} = M \times \{n\}$ and $\mathcal{I} = M \times \N.$ For every $\alpha \in \R$ and $N \geq 1$,
define
$$
m_{\alpha}(f , \Phi, \Lambda, \vep, N) := \inf_{\mathcal{G}}\Big\{\sum_{(x , n) \in \mathcal{G}}e^{-\alpha n + \phi_{n}(x)} \Big\},
$$
where the infimum is taken over every finite or enumerable families $\mathcal{G}\subset \cup_{n \geq N}\mathcal{I}_{n}$ such that the collection of sets $\{B(x , n , \vep) : (x , n) \in  \mathcal{G} \}$ cover $\Lambda$.
Since the sequence is monotone increasing in $N$,
{\color{black} the limit
$$
m_{\alpha}(f , \Phi, \Lambda, \vep) := \lim_{N \to +\infty}m_{\alpha}(f , \Phi, \Lambda, \vep, N)
$$
exists }
and
$
P_{\Lambda}(f , \Phi, \vep) := \inf\{\alpha : m_{\alpha}(f , \Phi, \Lambda, \vep) = 0 \} = \sup\{\alpha : m_{\alpha}(f , \Phi, \Lambda, \vep) = +\infty \}.
$
By Cao, Zhang and Zhao \cite{ZZC11}, the \emph{pressure} of $\Lambda$ is defined by the limit:
$$
P_{\Lambda}(f , \Phi) = \lim_{\vep \rightarrow 0}P_{\Lambda}(f , \Phi, \vep).
$$
If $\Lambda = M$ we have that $P_{\Lambda}(f , \Phi)$ corresponds to the \emph{topological pressure} of $f$
with respect to $\Phi$ and is denoted by $P_{\topp}(f , \Phi)$.
If we take a continuous potential $\phi$ we have that $P_{\Lambda}(f,\{\phi_{n}\}_{n})$, for $\phi_{n} = \sum_{i = 0}^{n-1}\phi \circ f^{i}$, is equal the
usual topological pressure of $\Lambda$ with respect to $f$ and $\phi$.
It follows of the definition of relative pressure that if $\Lambda_{1} \subset \Lambda_{2} \subset M$ we will have that $P_{\Lambda_{1}}(f, \Phi) \leq P_{\Lambda_{2}}(f, \Phi)$.
In the asymptotically additive context also we have the following variational principle:

\begin{proposition} \cite{Feng}
Let $M$ be compact metric space, $f : M \rightarrow M$ be continuous map and $\Phi = \{\phi_{n}\}$ a asymptotically additive sequence of potentials. Then
$$
P_{\topp}(f,\Phi) = \sup  \{ h_{\mu}(f) + \mathcal{F}_{\ast}(\Phi , \mu ) : \mu \;\mbox{is a f-invariant probability,}\; \mathcal{F}_{\ast}( \Phi, \mu ) \neq -\infty \},
$$
where the supremum is taken over all $f$-invariant probabilities $\mu$ and $\mathcal{F}_{\ast}( \Phi , \mu) = \lim_{n \to +\infty}\frac{1}{n}\int\phi_{n}d\mu$.
\end{proposition}

\subsection{Space of asymptotically additive sequences}

Given a compact metric space $ M $ let us define $\mathbb{A}:= \{\Psi = \{\psi_{n}\}_{n}: \Psi \text{ is asymptotically additive } \} $. The space $\mathbb{A} $ is clearly a vector space with a sum and product by a scalar defined naturally by $\{\psi_{1,n}\}_{n} + \{\psi_{2,n}\}_{n} := \{\psi_{1,n} + \psi_{2,n}\}_{n}$ and
$\lambda \cdot \{\psi_{1,n}\}_{n} := \{\lambda\psi_{1,n}\}_{n}$ for every
$\{\psi_{1,n}\}_{n} , \{\psi_{2,n}\}_{n} \in\mathbb{A}$ and $\lambda \in \R$.
On this vector space structure we shall consider the seminorm:
$||\{\psi_{n}\}_{n}||_\mathbb{A} := \limsup_{n \to \infty}\frac{1}{n}||\psi_{n}||_{\infty}$.
If necessary to consider a norm we can consider the space $\mathbb A$ endowed
with $\| \{\psi_{n}\}_{n} \|_{\mathbb A,0}:=\sup_{n\in \mathbb N} \frac{1}{n}||\psi_{n}||_{\infty}$ which clearly
satisfies $\| \{\psi_{n}\}_{n} \|_{\mathbb A} \le \| \{\psi_{n}\}_{n} \|_{\mathbb A,0}$ for every $ \{\psi_{n}\}_{n} \in {\mathbb A}$. For that reason we shall consider the continuity results with $\mathbb A$ endowed with the
weaker topology induced by the semi norm.
The balls of the seminorm $|| \cdot ||_\mathbb{A} $ form a basis for a topology on $\mathbb{A} $ that will not be metrizable because it is not Hausdorff. However $\mathbb{A}$ with the  aforementioned vector space structure
and with this topology is a 
locally convex topological vector space.
We shall consider $\mathbb{A} $ with this topology and the space of almost additive sequences
of observables with the natural induced topology.
\textcolor{black}{We endow the space $\mathcal M_1(M)$ of probability measures on $M$ with a distance $d$ that
induces the weak$^*$ topology and let $\mathcal M_1(f) \subset \mathcal M_1(M)$ denote the space of $f$-invariant probability measures.}

\begin{proposition}\label{admiss1}
Let $M$ be a compact metric space and $f : M \rightarrow M$ be a continuous map. Then
the following functions are continuous:
\begin{itemize}
 \item[i.] $\mathbb{A} \ni \Phi \mapsto P_{\topp}(f,\Phi)$;
 \item[ii.] $\mathcal{M}_{1}(f) \times\mathbb{A} \ni \textcolor{black}{(\Psi,\mu) \mapsto \mathcal{F}_{\ast}(\Psi,\mu)}$.
\end{itemize}
\end{proposition}

\begin{proof}
The first claim (i) is clear from the definition of topological pressure and the one of $||\cdot||_\mathbb{A}$.
Hence we are left to prove (ii). Given $\Psi_{1} = \{\psi_{1,n}\}_{n} \in\mathbb{A}$ and $\eta_{1} \in \mathcal{M}_{1}(f)$ arbitrary we will prove that
$(\mu , \Psi) \mapsto \mathcal{F}_{\ast}(\Psi,\mu)$ is continuous at
$(\Psi_{1}, \eta_{1})$. Let $\vep>0$ be small and fixed.

Since $\Psi_{1}\in \mathbb A$ there exists a continuous function $g_{\frac{\vep}{6}}$ and $n_{0} \in \N$ such that
$\frac{1}{n}||\psi_{1,n} - S_{n} g_{\frac{\vep}{6}}||_{\infty} < \frac{\vep}{6}$ for all $n \geq n_{0}$.
Moreover, there exists $\delta > 0$ such that if $d(\eta_{1} , \eta_{2}) < \delta$ then
$|\int g_{\frac{\vep}{6}} d\eta_{1} - \int g_{\frac{\vep}{6}} d\eta_{2}| < \frac{\vep}{6}$. Given $\Psi_{2} = \{\psi_{2,n}\}_{n} \in\mathbb{A}$ and $\eta_{2} \in \mathcal{M}_{1}(f)$ arbitrary in such a way that
 $||\Psi_{1} - \Psi_{2}||_\mathbb{A} < \frac{\vep}{6}$ and $d(\eta_{1} , \eta_{2})< \delta$ then there exists
 $n_{1} = n_{1}(\Psi_{2}, \eta_{2}) \geq n_{0}$ so that
 $
 \frac{1}{n_{1}}||\psi_{1,n_{1}} - \psi_{2,n_{1}}||_{\infty} < \frac{\vep}{6},
$
\; %
$
|\frac{1}{n_{1}}\int \psi_{2,n_{1}} d\eta_{2} - \mathcal{F}_{\ast}(\Psi_{2}, \eta_{2})| < \frac{\vep}{6}
$
and also
$
|\frac{1}{n_{1}}\int
\psi_{1,n_{1}} d\eta_{1} - \mathcal{F}_{\ast}(\Psi_{1}, \eta_{1})| < \frac{\vep}{6}.
$
Thus, given $\Psi_{2} = \{\psi_{2,n}\}_{n} \in\mathbb{A}$ and $\eta_{2} \in \mathcal{M}_{1}(f)$ such that $||\Psi_{1} - \Psi_{2}||_\mathbb{A} < \frac{\vep}{6}$ and
$d(\eta_{1} , \eta_{2})
< \delta$ we have that
\begin{align*}
|\mathcal{F}_{\ast}(\Psi_{1}, \eta_{1}) - \mathcal{F}_{\ast}(\Psi_{2},\eta_{2})|
    & \leq |\mathcal{F}_{\ast}(\Psi_{2}, \eta_{2}) - \int g_{\frac{\vep}{6}} d\eta_{2}| +
    \Big| \int g_{\frac{\vep}{6}} d\eta_{2} - \mathcal{F}_{\ast}(\Psi_{1},\eta_{1}) \Big| \\
    & \leq \Big|\frac{1}{n_{1}}\int S_{n_{1}} g_{\frac{\vep}{6}} d\eta_{2} -
            \frac{1}{n_{1}}\int \psi_{1,n_{1}} d\eta_{2} \Big| \\
    &+ \Big| \frac{1}{n_{1}}\int \psi_{1,n_{1}} d\eta_{2} - \mathcal{F}_{\ast}(\Psi_{2}, \eta_{2}) \Big| +
    \Big|\int g_{\frac{\vep}{6}} d\eta_{2} - \int g_{\frac{\vep}{6}} d\eta_{1} \Big| \\
    & +\Big| \int g_{\frac{\vep}{6}} d\eta_{1} - \mathcal{F}_{\ast}(\Psi_{1} , \eta_{1}) \Big|
 \end{align*}
 and so
 \begin{align*}
|\mathcal{F}_{\ast}(\Psi_{1}, \eta_{1}) - \mathcal{F}_{\ast}(\Psi_{2},\eta_{2})|
    & \leq  \frac{\vep}{3} + \Big|\frac{1}{n_{1}}\int \psi_{1,n_{1}} d\eta_{2} - \frac{1}{n_{1}}\int \psi_{2,n_{1}}
        d\eta_{2} \Big| \\
    &+ \Big| \frac{1}{n_{1}}\!\!\int \!\!\psi_{2,n_{1}} d\eta_{2} \!-\! \mathcal{F}_{\ast}(\Psi_{2} , \eta_{2}) \Big|
 +
     \Big| \frac{1}{n_{1}}\!\!\int \!\! \psi_{1,n_{1}} d\eta_{1} \!-\! \mathcal{F}_{\ast}(\Psi_{1} , \eta_{1}) \Big|\\
    &    + \Big| \int g_{\frac{\vep}{6}} d\eta_{1} \!-\! \frac{1}{n_{1}}\int \psi_{1,n_{1}} d\eta_{1} \Big|
\end{align*}
which is smaller than $\vep$.
This proves the continuity of $(\mu , \Psi) \mapsto \mathcal{F}_{\ast}(\Psi , \mu)$.
\end{proof}

Now we study some properties of the topological pressure in the case of repellers.

\begin{proposition}\label{admiss2}
Let $M $ be a Riemannian manifold, $ f: M \rightarrow M $ be a  $C^{1} $-map and
$ \Lambda \subset M $ be an isolated repeller such that $ f\mid_\Lambda $ is topologically mixing. Then:
 \begin{itemize}
  \item[i.] If $\Phi \in\mathbb{A}$ then $P_{\topp}(f,\Phi) = \lim_{\vep \to 0}P_{\topp}(f,g_{\vep})$ for any
        $(g_{\vep})_{\vep}$ admissible family for $\Phi$.
  \item[ii.] If $(\varphi_{\vep})_{\vep}$ is an admissible  family for $\Phi$ and $\mu_{\vep}$ is the unique equilibrium
    state for $f$ with respect to $\varphi_{\vep}$ then every accumulation point of $\mu_{\vep}$ is a equilibrium
    state for $f$ with respect to $\Phi$. In particular,  if there is a unique equilibrium state $\mu_\Phi$ for
    $f$ with respect to $\Phi$ then $\mu_\Phi= \lim_{\vep \to 0}\mu_{\vep}$.
 \end{itemize}
\end{proposition}

\begin{proof}
Property (i) follows from the corresponding item of Proposition~\ref{admiss1}. Now, since $ \Lambda $ is a repeller we have that $ \mu \rightarrow h_{\mu}(f) $ is upper semicontinuous and using the continuity of $(\mu,\Phi)\mapsto
\cF_*(\mu,\Phi)$ we conclude that  every accumulation point of $ \mu_{\vep} $ is equilibrium state for $ f $ with
respect to $ \Phi $. Using the compactness of the space of invariant probabilities, if there exists a unique equilibrium
$ \mu_\Phi $ for $ f $ with respect to $ \Phi $ then the convergence $\mu_\Phi = \lim_{\vep \to 0}\mu_{\vep}$ holds. This finishes the proof of the proposition.
\end{proof}


\subsection{Free energy function and Legendre transforms}\label{sec:free}

We are interested in the regularity of the rate function in the large deviations principles
obtained in \cite{Va13}. Since there exists no direct functional analytic approach using Perron-Frobenius operators, in order to inherit some properties from the classical {\color{black}thermodynamic} formalism we will use
the approximation by admissible families of H\"older continuous functions.

The next result allows us to define the Legendre transform of $ \Psi $ in terms of the Legendre transform associated to any approximating admissible family.
For every almost additive sequence of potentials $ \Phi $ satisfying the bounded distortion condition we denote by $ \mu_{\Phi} $ the unique equilibrium state of $f$ with respect to
$ \Phi $ (for the existence of $\mu_{\Phi}$ see \cite{Barreira}). For $ \Phi, \Psi \in \mathbb A$ consider the free energy function given by
$\mathcal{E}_{f , \Phi,\Psi}(t) : = P_{\topp}(f , \Phi + t\Psi) - P_{\topp}(f , \Phi)$
for all $ t \in \R $. Observe that
$$
\mathcal{E}_{f , \Phi,\Psi} = \lim_{\vep \to 0} \mathcal{E}_{f , \varphi_{\vep},g_{\vep}},
$$
where $\{\varphi_{\vep}\}_{\vep}$ and $\{g_{\vep}\}_{\vep}$ are any admissible families for $\Phi$ and
$\Psi$ respectively. In fact, the pressure function is continuous in the set of all asymptotically additive
 sequences and so this limit does not depend on the sequence of approximating families and we may
 take H\"older continuous representatives (admissible families).

Assume $ \Phi,\Psi $ are almost additive sequences of potentials satisfying bounded distortion
condition so that $ \Psi $ is not cohomologous to a constant and let
$ \{g_{\vep}\}_{\vep} $ be an admissible family for $ \Psi $ not cohomologous to a constant and
$ (\varphi_{\vep})_{\vep} $ be an admissible family for $ \Phi $. Then it makes sense to define
for every $ \vep \in (0, 1) $
$$
 I_{f, \varphi_{\vep}, g_{\vep}}(t)=\sup_{s\in\mathbb R} \Big( st - \mathcal{E}_{f , \varphi_{\vep},g_{\vep}}(s) \Big)
$$
as the Legendre transform of $\mathcal{E}_{f , \phi_{\vep},g_{\vep}}$.
Since each $ g_{\vep} $ is not cohomologous to a constant then {\color{black}the previous function is defined
over a open interval} and the variational property yields
\begin{equation}\label{eq:freevar}
I_{f,\varphi_{\vep},g_{\vep}}(\mathcal{E}'_{f,\varphi_{\vep},g_{\vep}}(t))
    = t \mathcal{E}'_{f,\varphi_{\vep},g_{\vep}}(t) - \mathcal{E}_{f,\varphi_{\vep},g_{\vep}}(t)
\end{equation}
for all $\vep \in (0 , 1)$ and $t \in \R$, and $\mathcal{E}_{f,\varphi_{\vep},g_{\vep}}$ is strictly convex (see e.g. \cite{DK,BCV13}).
Recalling that
$$
\mathcal{E}_{f , \varphi_\vep,g_\vep} = P_{\topp}(f , \varphi_\vep + t g_\vep) - P_{\topp}(f , \varphi_\vep)
    \quad \text{ and }\quad
    \mathcal{E}'_{f,\varphi_{\vep},g_{\vep}}(t)=\int g_\vep \;d\mu_{\varphi_\vep+t g_\vep}
$$
it follows from Propositions \ref{admiss1} and \ref{admiss2} that for every $t\in\R$
$$
\lim_{\vep \to 0}\mathcal{E}'_{f,\varphi_{\vep},g_{\vep}}(t)
    = \lim_{\vep \to 0} \int g_\vep \; d\mu_{\varphi_\vep+t g_\vep}
    = \mathcal{F}_{\ast}(\Psi , \mu_{f,\Phi + t \Psi}),
$$
that
$
\lim_{\vep\to 0} \mathcal{E}_{f,\varphi_{\vep},g_{\vep}}(t) = P_{\topp}(f , \Phi + t \Psi) - P_{\topp}(f,\Phi),
$
and that these convergences are uniform in compact sets.
Observe that the Legendre transform $I_{f,\varphi_{\vep},g_{\vep}}$ is well
defined in the open interval
$J_\vep=(\inf_t \mathcal{E}'_{f,\varphi_{\vep},g_{\vep}}(t), \sup_{t} \mathcal{E}'_{f,\varphi_{\vep},g_{\vep}}(t)) $.
Thus, we can now define the \emph{Legendre transform}
$$
I_{f , \Phi, \Psi}(s) := \lim_{\vep \to 0}I_{f,\varphi_{\vep},g_{\vep}}(s)
$$
for any
$
s \in J_{\Phi,\Psi}
    :=\big(\inf_{t \in \R}\mathcal{F}_{\ast}(\Psi , \mu_{\Phi + t \Psi})
    \;,\; \sup_{t \in \R}\mathcal{F}_{\ast}(\Psi , \mu_{\Phi + t \Psi}) \big).
$
Note that the previous limiting function and interval do
not depend of the chosen families
$\{\varphi_{\vep}\}_{\vep}$ and $\{g_{\vep}\}_{\vep}$.
A priori $J_{\Phi,\Psi}$ could be a degenerate interval. However the next lemma assures that the closure of this interval is exactly the spectrum of $\Psi$ and, in particular, $J_{\Phi,\Psi}$ is a degenerate interval if and only if $\Psi$ is cohomologous to a constant.

\begin{lemma}
Given $\Psi , \Phi \in \mathbb{A}$ we have that
$$
\Big[\inf_{t \in \R}\mathcal{F}_{\ast}(\Psi , \mu_{\Phi + t \Psi}) \;,\; \sup_{t \in \R}\mathcal{F}_{\ast}(\Psi , \mu_{\Phi + t \Psi}) \Big]
=  \{ \mathcal{F}_{\ast}(\Psi , \eta)  : \eta \in \mathcal{M}_{1}(f) \}
$$
and the interval is degenerate if only if $\Psi$ is cohomologous to a constant.
\end{lemma}

\begin{proof}
Let $(\varphi_{\vep})_\vep$ and $(g_{\vep})_\vep$ be any admissible sequences  for $\Phi$ and $\Psi$
respectively. We have that $\inf_{t \in \R}\mathcal{F}_{\ast}(\Psi , \mu_{\Phi + t \Psi})=\lim_{\vep \to 0} \inf_{t \in \R} \mathcal{E}'_{f,\varphi_{\vep},g_{\vep}}(t)$
(analogously for the supremum) and so
$$
\bigcap_{\vep_{0} \in (0 , 1)}\bigcup_{\vep \geq \vep_{0}}\Big(\inf_{t \in \R}\mathcal{E}'_{f,\varphi_{\vep},g_{\vep}}(t) \;,\; \sup_{t \in \R}\mathcal{E}'_{f,\varphi_{\vep},g_{\vep}}(t) \Big) =
\Big(\inf_{t \in \R}\mathcal{F}_{\ast}(\Psi , \mu_{\Phi + t \Psi}) \;,\; \sup_{t \in \R}\mathcal{F}_{\ast}(\Psi , \mu_{\Phi + t \Psi}) \Big)
$$
and
$$
\bigcap_{\vep_{0} \in (0 , 1)}\bigcup_{\vep \geq \vep_{0}}\Big(\inf_{\eta \in \mathcal{M}_{1}(f)} \int g_{\vep} d\eta \;,\; \sup_{\eta \in \mathcal{M}_{1}(f)} \int g_{\vep} d\eta \Big) =
\Big(\inf_{\eta \in \mathcal{M}_{1}(f)} \mathcal{F}_{\ast}(\Psi , \eta) \;,\; \sup_{\eta \in \mathcal{M}_{1}(f)} \mathcal{F}_{\ast}(\Psi , \eta) \Big).
$$

Let us first prove the result in the additive setting, that is, assuming there are $g,\phi  \in C(M , \R)$ such that
$\varphi_n =S_n \,\phi$ and $\psi_n =S_n \,g$. If this is the case, using the weak* continuity of the equilibrium states
with respect to the potential, the image of the function $T_g: \R \to \R$ given by $t \mapsto \int g \, d\mu_{\phi + tg}$
is an interval.
In addition, given $\eta \in \mathcal{M}_{1}(f)$ and
$t > 0$ we have by the variational principle
$$
h_{\eta}(f) + \int (\phi +t g) \; d\eta
    \leq h_{\mu_{\phi +tg}}(f) + \int (\phi + t g) \; d\mu_{\phi +tg}
$$
and so, dividing by $t$ in both sides and making $t$ tend  to infinity in the expression
$$
\frac{1}{t}h_{\eta}(f) + \frac{1}{t}\int \phi \;d\eta+ \int g \;d\eta
    \leq
\frac{1}{t} h_{\mu_{\phi +tg}}(f) + \frac{1}{t}\int \phi \;d\mu_{\phi +tg} + \int g \; d\mu_{\phi +tg},
$$
we get that
$
\ds\int g d\eta
    \leq \limsup_{t \to +\infty} \int g \;d\mu_{\phi +tg}
    = \int g \;d\mu_*
$
for an $f$-invariant probability $\mu_*$ properly chosen as accumulation point of $(\mu_{\phi +tg} )_t$. This proves
that
$\sup_{\eta\in \cM_1(f)} \int g \; d\eta=\limsup_{t \to +\infty} \int g \;d\mu_{\phi +tg}$.
Proceeding analogously with $-g$ replacing $g$ it follows that
$\inf_{\eta\in \cM_1(f)} \int g \; d\eta
    = \liminf_{t \to -\infty} \int g \;d\mu_{\phi +tg}$
    and
\begin{equation}\label{eq.int.add}
\Big[\inf_{t \in \R} \int g \; d \mu_{\phi + t g} \;,\; \sup_{t \in \R} \int g \; d \mu_{\phi + t g} \Big]
    =  \Big\{ \int g \; d \eta  : \eta \in \mathcal{M}_{1}(f) \Big\}.
\end{equation}

Now, to deal with the general non-additive setting, replacing $g$ by $g_\vep$ and also $\phi$ by $\varphi_\vep$
in equation~\eqref{eq.int.add}, and taking the limit as $\vep$ tends to zero it follows that
\begin{equation*}
\Big[\inf_{t \in \R}\mathcal{F}_{\ast}(\Psi , \mu_{\Phi + t \Psi}) \;,\; \sup_{t \in \R}\mathcal{F}_{\ast}(\Psi , \mu_{\Phi + t \Psi}) \Big]
    =  \Big\{ \cF_*(\Psi,\eta) : \eta \in \mathcal{M}_{1}(f) \Big\}.
\end{equation*}
as claimed. This finishes the first part of the proof of the lemma.

Finally, by \cite[Lemma~2.2]{ZZC11} we get
$\inf_{t \in \R}\mathcal{F}_{\ast}(\Psi , \mu_{\Phi + t \Psi}) = \sup_{t \in \R}\mathcal{F}_{\ast}(\Psi , \mu_{\Phi + t \Psi}) $ if
and only if $\frac{\psi_n}{n}$ converges uniformly to a constant, that is, $\Psi$ is cohomologous to a constant.
This finishes the proof of the lemma.
\end{proof}

\begin{remark}
It is not hard to check also that there exists a constant $C>0$ (depending only on $f$) so that
$\frac{P(f,\Phi+t\Psi)}{t}=\mathcal{F}_*(\Psi,\mu_{\Phi+t\Psi}) \pm \frac{C}{t}$ and, consequently, the previous interval
is characterized as the interval of limiting slopes for the pressure function $t\mapsto P(f,\Phi+t\Psi)$.
\end{remark}

\subsection{Proof of Theorem~\ref{legtrans}}

 Let $\Phi, \Psi$ be almost additive sequences of H\"older continuous
 potentials satisfying the bounded distortion condition so that
 $\Psi$ be is not cohomologous to a constant. In particular $t\mapsto \mathcal{F}_{\ast}( \Psi, \mu_{\Phi+t\Psi})$
 is not a constant function. Moreover, the Legendre transform of the free energy function
 $I_{f,\Phi,\Psi}$ (defined in the previous section) is well defined in an open neighborhood of the mean
 $\mathcal{F}_{\ast} (\Psi, \mu_{\Phi})$.

 Let $ (\varphi_{\vep})_{\vep} $ and $(g_{\vep} )_{\vep} $ be any admissible families for $ \Phi $ and  $ \Psi $, respectivelly. It follows from equation~\eqref{eq:variational} that
$
I_{f,\varphi_{\vep},g_{\vep}}(\mathcal{E}'_{f,\varphi_{\vep},g_{\vep}}(t))
    = t \mathcal{E}'_{f,\varphi_{\vep},g_{\vep}}(t) - \mathcal{E}_{f,\varphi_{\vep},g_{\vep}}(t)
$
for every $t\in\R$
{\color{black}and so, letting $\vep$ converge to}
zero, we obtain that
\begin{equation}\label{eq:variatA}
I_{f,\Phi,\Psi}(\mathcal{E}'_{f,\Phi,\Psi}(t)) = t \mathcal{E}'_{f,\Phi,\Psi}(t) - \mathcal{E}_{f,\Phi,\Psi}(t),
\end{equation}
which proves (i).
Now, since $I_{f,\varphi_{\vep},g_{\vep}}$ is a non-negative convex function for all $\vep \in (0 , 1)$  and is
pointwise convergent to $I_{f,\Phi,\Psi}$ this is also a non-negative convex function.
Clearly, given any interval $(a , b) \subset \R$ not containing
$\mathcal{F}_{\ast}(\Psi,\mu_{\Phi})$ then we know that
$$
\inf_{s \in (a , b)} I_{f,\varphi_{\vep},g_{\vep}}(s) = \min\{I_{f,\varphi_{\vep},g_{\vep}}(a) , I_{f,\varphi_{\vep},g_{\vep}}(b)\},
$$
so the same property will be valid for the limit function $I_{f,\Phi,\Psi}$, which proves (ii).

Let us prove \textcolor{black}{ (iii)}, that is, to establish the variational formula
  $$
  I_{f,\Phi,\Psi}(s)
    = \inf_{\eta \in \mathcal{M}_{1}(f)}
    \{P_{\topp}(f, \Phi) - h_{\eta}(f) - \mathcal{F}_{\ast}(\Phi, \eta ) : \mathcal{F}_{\ast}(\Psi, \eta ) = s\}
 $$
for the rate function. The equality is clearly satisfied when $ s = \mathcal{F}_{\ast}(\Psi, \mu_{\Phi}) $ by uniqueness
of the equilibrium state and Proposition \ref{admiss2}. Hence we are reduced to the case where $ s \neq \mathcal{F}_{\ast}(\Psi, \mu_{\Phi}) $.
From the additive case we already know that for all
$
s\in \big(\inf_{t \in \R}\mathcal{F}_{\ast}(\Psi , \mu_{\Phi + t \Psi})
    \;,\; \sup_{t \in \R}\mathcal{F}_{\ast}(\Psi , \mu_{\Phi + t \Psi}) \big)
$
$$
I_{f,\varphi_{\vep},g_{\vep}}(s)
    = \inf_{\eta \in \mathcal{M}_{1}(f)}
    \big\{P_{\topp}(f, \varphi_{\vep}) - h_{\eta}(f) - \int \varphi_{\vep} d\eta : \int g_{\vep} d\eta = s\big\}
$$
and for every small $\vep$. We will use an auxiliary lemma.

\begin{lemma}
For every $s$ in the interior of $J := \{ \mathcal{F}_{\ast}(\Psi , \eta)  : \eta \in \mathcal{M}_{1}(f) \}$,
$$
\lim_{\vep \to 0}\sup_{\eta \in \mathcal{M}_{1}(f)} \{h_{\eta}(f) + \int \varphi_{\vep} d\eta : \int g_{\vep} d\eta \!=\! s\}
=
\!\!\!\sup_{\eta \in \mathcal{M}_{1}(f)} \!\{h_{\eta}(f) + \mathcal{F}_{\ast}(\Phi, \eta ) : \mathcal{F}_{\ast}(\Psi, \eta ) = s\}.
$$
\end{lemma}

\begin{proof}
We will use the continuity of $\cF_*(\Phi,\mu)$ in both coordinates. Let $s \in J$ be fixed
and {\color{black} take $\eta_{1} \in \mathcal{M}_{1}(f)$ with $s=\mathcal{F}_{\ast}(\Psi, {\eta_{1}})$.}
Consider an admissible family $(g_{\vep})_\vep$ for $\Psi$ not cohomologous a to constant. We may assume
without loss of generality that \textcolor{black}{$\int g_{\vep}d\eta_1 = s$} for $\vep$ small (otherwise just use the
admissible family $(\tilde g_\vep)_\vep$ given by \textcolor{black}{$\tilde g_\vep:=g_\vep + s- \int g_\vep \,d\eta_1$} which is also
not cohomologous to a constant). In particular,
$$
\big\{
\eta \in \mathcal{M}_{1}(f) \colon \int g_\vep \;d\eta =s
    \text{ for all $\vep$ small }
\big\}
$$
is a closed, non-empty set in $\mathcal{M}_{1}(f) $, hence compact. Using the compactness and
upper semi-continuity of the metric entropy function there exists $\eta_{\vep} \in \mathcal{M}_{1}(f)$ such that
$\int g_{\vep}d\eta_{\vep} = s$ and
$$
h_{\eta_{\vep}}(f) + \int \varphi_{\vep} d\eta_{\vep}
    =  \sup_{\eta \in \mathcal{M}_{1}(f)} \{h_{\eta}(f) + \int \varphi_{\vep} d\eta : \int g_{\vep} d\eta = s\}.
$$
Take $\tilde{\eta} \in \mathcal{M}_{1}(f)$ be an accumulation point of $(\eta_{\vep})_\vep$ and assume for
simplicity that $\eta_\vep \to \tilde\eta$ as $\vep$ tends to zero. Then Proposition~\ref{admiss1} yields that
$\lim_{\vep \to 0}\int\varphi_{\vep} d\eta_{\vep} = \mathcal{F}_{\ast}( \Phi, \tilde{\eta} )$ and
$\lim_{\vep \to 0}\int g_{\vep} d\eta_{\vep} = \mathcal{F}_{\ast}(\Psi, \tilde{\eta} ) = s$.
Using once more the upper semicontinuity of the metric entropy function
\begin{align*}
\lim_{\vep \to 0}\sup_{\eta \in \mathcal{M}_{1}(f)} \{h_{\eta}(f) + \int \varphi_{\vep} d\eta : \int g_{\vep} d\eta = s\}
                 & = \lim_{\vep \to 0} \big\{ h_{\eta_{\vep}}(f) + \int \varphi_{\vep} d\eta_{\vep} \big\}\\
                 & \leq h_{\tilde{\eta}}(f) + \mathcal{F}_{\ast}(\Phi, \tilde{\eta}) \\
                 & \leq \sup_{\eta \in \mathcal{M}_{1}(f)} \{h_{\eta}(f) + \mathcal{F}_{\ast}(\Phi,\eta) :
                 \mathcal{F}_{\ast}( \Psi, \eta) = s\}.
\end{align*}

To prove the other inequality, let $\tilde{\eta} \in \mathcal{M}_{1}(f)$ be that attains the supremum in the right hand side above,
that is, so that $s=\mathcal{F}_{\ast}(\Psi, \tilde{\eta})$ and
\begin{align*}
\sup_{\eta \in \mathcal{M}_{1}(f)} \{h_{\eta}(f) + \mathcal{F}_{\ast}(\Phi, \eta) : \mathcal{F}_{\ast}(\Psi, \eta) = s\}
    & = h_{\tilde{\eta}}(f) + \mathcal{F}_{\ast}(\Phi,\tilde{\eta} )
\end{align*}
Let $\delta>0$ be fixed and arbitrary. By Proposition \ref{admiss1} there exists $\vep_{\delta} > 0$ such that
$\int g_{\vep} d\tilde{\eta} \in (s - \delta , s + \delta)$ for all  $0<\vep < \vep_{\delta}$. In particular, using the characterization of rate function $I_{f,\varphi_{\vep},g_{\vep}}(\cdot)$ given by \cite{Young}
\begin{align*}
 h_{\tilde{\eta}}(f) + \int \varphi_{\vep} d\tilde{\eta}
              & \leq \sup_{\eta \in \mathcal{M}_{1}(f)} \{h_{\eta}(f) + \int \varphi_{\vep} d\eta : \int g_{\vep} d\eta \in (s -\delta , s +\delta)\} \\
              & = -\inf_{t \in (s-\delta , s +\delta)} I_{f,\varphi_{\vep},g_{\vep}}(t) + P_{\topp}(f , \varphi_{\vep}) \\
\end{align*}
for every  $0<\vep < \vep_{\delta}$. Taking the limit as $\vep\to0$ in both sides of the inequality
and using the convexity of the Legendre transform
\begin{align*}
\sup_{\eta \in \mathcal{M}_{1}(f)} \{h_{\eta}(f) + \mathcal{F}_{\ast}(\Phi, \eta ) : \mathcal{F}_{\ast}( \Psi,\eta ) = s\}
                                   & = \lim_{\vep \to 0} ( h_{\tilde{\eta}}(f) + \int \varphi_{\vep} d\tilde{\eta} ) \\
              & \leq \lim_{\vep \to 0} \;
                (P_{\topp}(f , \varphi_{\vep}) -\inf_{t \in (s-\delta , s +\delta)} I_{f,\varphi_{\vep},g_{\vep}}(t))  \\
              & = P_{\topp}(f , \Phi) -\min\{I_{f,\Phi,\Psi}(c-\delta) , I_{f,\Phi,\Psi}(c + \delta)\}
\end{align*}
Since the rate function is continuous, taking $\delta$ tend to zero it follows
\begin{align*}
\sup_{\eta \in \mathcal{M}_{1}(f)} \{h_{\eta}(f) + & \mathcal{F}_{\ast}(\Phi, \eta ) : \mathcal{F}_{\ast}( \Psi, \eta ) =s\} \\
                                                    & \leq  - I_{f,\Phi,\Psi}(c) + P_{\topp}(f , \Phi)\\
                                                    & = \lim_{\vep \to 0} \{-I_{f,\varphi_{\vep},g_{\vep}}(s) + P_{\topp}(f , \varphi_{\vep})\} \\
                                               & = \lim_{\vep \to 0} \sup_{\eta \in \mathcal{M}_{1}(f)} \{h_{\eta}(f) + \int \varphi_{\vep} d\eta : \int g_{\vep} d\eta = s\}.
\end{align*}
This finishes the proof of the lemma.
\end{proof}
Now, item (iii) is just a consequence of the previous lemma together with the fact
that $I_{f , \Phi, \Psi}(s) := \lim_{\vep \to 0}I_{f,\varphi_{\vep},g_{\vep}}(s)$.

We are left to prove property (iv). It follows from item (iii) that $I_{f,\Phi,\Psi}(s)=0$ if and only if
$s=\mathcal{F}_{\ast}(\Psi,\mu_{\Phi} )$. It remains to prove that  $I_{f,\Phi,\Psi}$ is strictly convex in a small neighborhood
of $\mathcal{F}_{\ast}(\Psi,\mu_{\Phi} )$. {\color{black} The proof is by contradiction and assuming $\mathcal{F}_{\ast}(\Psi,\mu_{\Phi} )=0$ below causes no
loss of generality and simplifies the notation.
If this was not the case, and using that $I_{f,\Phi,\Psi}$ is convex, it is not
hard to check that either:
(a) there exists an open interval around $0$ where $I_{f,\Phi,\Psi}$ is constant, or
(b) there exists $c\in \mathbb R$ and an open interval $J$ with $0$ as an endpoint so that
	$I_{f,\Phi,\Psi} (s)= c s $ for every $s\in J$ (i.e. the function is affine).
Since $I_{f,\Phi,\Psi}(s) = 0$ if and only if $s =0$ then
case (a) clearly contradicts the uniqueness of the equilibrium state $\mu_{\Phi}$.
 In case (b) it follows from \eqref{eq:variatA} that
$$
c \,  \mathcal{E}'_{f,\Phi,\Psi}(t)
	= t \mathcal{E}'_{f,\Phi,\Psi}(t) - \mathcal{E}_{f,\Phi,\Psi}(t),
$$
for every $t\in \mathbb R$ with $\mathcal{E}'_{f,\Phi,\Psi}(t) \in J $. Then, the
(unique) solution of the previous non-autonomous linear differential equation is $\mathcal{E}_{f,\Phi,\Psi}(t) = a( t-c)$ for some
$a\in \mathbb R$.
If $a\ne 0$ then \eqref{eq:defEt} implies that $t\mapsto P_{\topp}(f , \Phi + t\Psi)$ is affine for an open interval $J' \subset \mathbb R$
having $0$ as an endpoint. Since for every $f$-invariant probability measure $\eta$, the line $q\mapsto h_{\eta}
+ \mathcal{F}_*(\Phi + q \Psi,\eta)$ is a subdiferential for topological pressure, the semi-continuity of the entropy function
together with the later expression implies that any accumulation point $\mu$ for $\mu_{\Phi+t \Psi}$ as $t\to 0$ is an
equilibrium state for $f$ with respect to $\Phi$ which satisfies $\mathcal{F}_*(\Psi, \mu)= a \neq 0 = \mathcal{F}_*(\Psi, \mu_\Phi)$ which contradicts uniqueness of equilibrium states.
If $a=0$ then there exists an open interval $J'$ so that the unique equilibrium state for $f$ with respect
to $\Phi+t\Psi$ is the same for all $t\in J'$. Then it follows from \cite{YZhou}
(following the ideas from \cite[Proposition~4.5]{BowenLNM}) that the sequences
$\Phi$ and $\Phi+t\Psi$ of almost-additive observables are cohomologous for every $t\in J'$,
which implies that $\Psi$ is cohomologous to a constant. Since the later cannot occur, since we assume
that $\Psi$ is not cohomologous to a constant, this completes the proof of (iv).
}
This finishes the proof of Theorem~\ref{legtrans}.

\section{Multifractal analysis of irregular sets}\label{sec:frac}

This section is devoted to the proof of our multifractal analysis results.

\subsection{Proof of Theorem~\ref{thm:B}}

Let $M$ be a compact metric space, $f: M\to M$ be a continuous map, $\Phi = \{\phi_{n}\}$ be an almost additive sequence of potentials with $P_{\topp}(f,\Phi) > -\infty$. By assumption, the unique equilibrium state $\mu_\Phi$
of $f$ with respect to $\Phi$ 
is a weak Gibbs measure.
Given $J\subset \mathbb R$ and $n\ge 1$ set  $X_{J,n}=\{x \in M : \frac1n \psi_{n}(x) \in J\}$.

\begin{lemma}\label{lemma:aux3}
Assume that $\Psi = \{\psi_{n}\}$ is  a sequence of observables  that satisfies at least one of the following
properties:
\begin{itemize}
\item[(a)] $\Psi$ is asymptotically additive or;
\item[(b)] $\Psi$ is a subadditive sequence such that
        \begin{itemize}
        \item[i.] it satisfies the weak Bowen condition;
        \item[ii.] $\inf_{n \geq 1}\frac{\psi_{n}(x)}{n} > -\infty$ for every $x \in M$; and
        \item[iii.] the sequence $\{\frac{\psi_{n}}{n}\}$ is equicontinuous.
        \end{itemize}
\end{itemize}
Then $ \Psi $ satisfies the tempered distortion condition.
In particular, given $J \subset \mathbb R$ be a closed set and $\delta> 0$ there exists $\vep_\delta>0$ such that if $0< \vep < \vep_\delta$
then there exists $N=N_{\delta,\epsilon} \in \N$ so that
$B(x ,n,  \vep) \subset X_{J_{\delta}, n}$ for all $n \geq N$ and every $x \in X_{ J, n}$.
\end{lemma}

\begin{proof}
The tempered distortion condition is clear for sequences of observables satisfying the weak Bowen condition and also holds for asymptotically additive sequences (c.f. \cite[Lemma~2.1]{ZZC11}).

Let us prove now the second part of the lemma. Given $\delta>0$, by the tempered distortion condition there is $\vep_\delta>0$ such that
$\lim_{n \to \infty}\gamma_{n}(\psi,\vep) < \delta n$ for all
$0 <\vep <\vep_\delta$. So, given $0 <\vep <\vep_\delta$
there exists a large $N=N_{\delta,\vep} \in \N$ such that if
$n \geq N$ we have $\gamma_{n}(\psi,\vep) \leq \delta n$. So, if $0<\vep<\vep_\delta$, $n \geq N$ and $x \in X_{ J, n}$, $y \in B(x ,n, \vep)$ then
$$
 \frac{\psi_{n}(x)}{n} -\frac{\gamma_{n}(\psi,\vep)}{n}
    \leq \frac{\psi_{n}(y)}{n}
    \leq  \frac{\psi_{n}(x)}{n} + \frac{\gamma_{n}(\psi,\vep)}{n}
$$
and, consequently,
$$
\frac{\psi_{n}(x)}{n} -\delta
    \leq \frac{\psi_{n}(y)}{n}
    \leq \frac{\psi_{n}(x)}{n} + \delta
$$
meaning that $y \in X_{J_{\delta}, n}$.
This finishes the proof of the lemma.
\end{proof}

We can now proceed with the proof of Theorem~\ref{thm:B}.
Assuming that $\overline{X}_{J}  \neq \emptyset$, we shall prove that
$P_{\overline{X}_{J}}(f , \Phi) \leq P_{\topp}(f,\Phi) - L_{J_{\delta},\mu_\Phi}$.
If $L_{J_{\delta},\mu_\Phi}=0$ there is nothing to prove so we assume without loss of generality that
$L_{J_{\delta},\mu_\Phi}>0$.
For our purpose it is enough to prove that for every $\alpha > P_{\topp}(f,\Phi) - L_{J_{\delta},\mu_\Phi}$,
given $\epsilon > 0$ and $N \in \N$ there exists
$\mathcal{G}_N \subset \bigcup_{n \geq N}\mathcal{I}_n$ satisfying the covering property $\ds\bigcup_{(x,n)\in \mathcal{G}_N} B(x,n,\epsilon) \supset \overline{X}_{I} $ and also
$
\ds\sum_{(x , n) \in \mathcal{G}_N} e^{-\alpha n +\phi_{n}(x)} \leq a(\epsilon)< \infty.
$
Let $\alpha > P_{\topp}(f,\Phi) - L_{J_{\delta},\mu_\Phi} $ and $0<\vep<\vep_\delta$ fixed, we take
 $\zeta >0$ small so that $\alpha > P_{\topp}(f,\Phi) - L_{J_{\delta},\mu_\Phi} +\zeta$.
{\color{black}Since $\mu_\Phi$ is a weak Gibbs measure,}
there exists $N_{0} \geq N_{\delta,\epsilon}$ such that $K_{n}(\vep) \leq e^{\frac{\zeta}{4}n}$,
$K_{n}(\frac{\vep}{2}) \leq e^{\frac{\zeta}{4}n}$ and
\begin{equation*}
\mu_\Phi \Big(\{x \in M : \frac{1}{n}\psi_{n}(x) \in J_{\delta} \}\Big)
    \leq e^{-(L_{J_{\delta},\mu_\Phi}-\frac{\zeta}{2}) n}
\end{equation*}
for all $n\ge N_{0}$. There is no loss of generality in supposing that $N \geq N_{0}$.
Given $N \geq N_{0}$ and $x \in \overline{X}_{J}$ take $m(x) \geq N$ so that $x \in X_{J_{\frac{\delta}{2}}, m(x)}$ and consider
$\mathcal{G}_{N} :=\{(x,m(x)): x \in \overline{X}_{J}\}$. Now, let $\hat{\mathcal{G}}_N \subset \mathcal{G}_N$ be a
maximal $(\ell,\vep)$-separated set. In particular if $(x, \ell)$ and $(y,\ell)$ belong the
$\hat{\mathcal{G}}_N$ then $B(x,\ell,\frac{\vep}{2}) \cap B(x, \ell, \frac{\vep}{2}) = \emptyset$.
Hence, for $0 < \vep < \vep_\delta$ given by Lemma~\ref{lemma:aux3}, using the Gibbs property of $\mu_\Phi$
we deduce that
\begin{align*}
\sum_{(x,m(x))\in \hat{\mathcal{G}}_N}   e^{-\alpha m(x) +\phi_{m(x)}(x)}
&=\sum_{(x,m(x))\in \hat{\mathcal{G}}_N}    e^{(P-\alpha)m(x)}e^{-P m(x) +\phi_{m(x)}(x)}\\
&\leq \sum_{(x,m(x))\in \hat{\mathcal{G}}_N}   e^{(P-\alpha)m(x)}K_{m(x)}(\vep) \, \mu_\Phi(B(x,m(x),\vep))
\end{align*}
Now, we write $\hat{\mathcal{G}}_N=\cup_{\ell\ge 1} \hat{\mathcal{G}}_{\ell,N}$ with the level sets
$\hat{\mathcal{G}}_{\ell,N}:=\{(x,\ell)\in \hat{\mathcal{G}}_{N}\}$.
By Lemma~\ref{lemma:aux3} each dynamical ball $B(x,\ell,\vep)$ is contained in
$X_{I_{\delta},\ell}$. Thereby, using that $\mu_\Phi(B(x,m(x),\vep)) \leq K_{m(x)}(\vep) K_{m(x)}(\vep/2)
\mu_\Phi(B(x,m(x),\vep/2)$
then
 \begin{align*}
\sum_{(x,m(x))\in \hat{\mathcal{G}}_N}   e^{-\alpha m(x) +\phi_{m(x)}(x)}
        & \leq  \sum_{(x,m(x))\in \hat{\mathcal{G}}_N}K_{m(x)}(\vep) e^{(P-\alpha)(m(x))} \mu_\Phi(B(x,m(x),\vep)) \\
        & = \sum_{\ell \geq N}K_{\ell}(\vep)e^{(P-\alpha)\ell}
            \sum_{x \in \hat{\mathcal{G}}_{N,\ell}}  \mu_\Phi(B(x,\ell,\vep)) \\
        & \leq
      \sum_{\ell\geq N}K_{\ell}(\vep)K_{\ell}(\frac{\vep}{2})e^{(P-\alpha)\ell}
      \sum_{x \in \hat{\mathcal{G}}_{N,\ell}}  \mu_\Phi(B(x,\ell,\vep/2)) \\
      & \leq
      \sum_{\ell\geq N}K_{\ell}(\vep)K_{\ell}(\frac{\vep}{2})e^{(P-\alpha) \ell} \, \mu_\Phi(X_{J_\delta,\ell}) \\
      & \leq
      \sum_{\ell\geq N}e^{(P-\alpha-L_{J_{\delta},\mu_\Phi}+\zeta) \ell}
\end{align*}
that is finite and independent of the choice of $N$. This proves that for any closed interval $J \subset \R$ and any small $\delta > 0$ it follows that
$
P_{\underline{X}_{J}}(f , \Phi) \leq P_{\overline{X}_{J}}(f , \Phi)
    \leq P_{\topp}(f , \Phi) - L_{J_{\delta},\mu_\Phi}
    \leq P(f , \Phi),
$
 proving the theorem.

\begin{remark}\label{veryweak}
Let us mention that the argument of Theorem \ref{thm:B} proving that for any closed interval
$J \subset \R$ and any small $\delta > 0$,
\begin{equation}\label{eq:veryweakest}
P_{\underline{X}_{J}}(f , \Phi) \leq  P_{\topp}(f , \Phi) -L_{J_{\delta},\mu_\Phi} \leq P(f , \Phi)
\end{equation}
carries under the weaker Gibbs condition ~\eqref{eq:veryweak}.
Taking into account the difficulty that the moments where the Gibbs property occurs may depend on the point
justifies the fact that the estimate \eqref{eq:veryweakest} holds for the set $\underline{X}_{J}$.
Since the proof of  this fact is similar to the the one of Theorem \ref{thm:B} we give only a sketch of proof with main ingredients. In fact, by ~\eqref{eq:veryweak} there is $ \vep_0 >  0$ such that:
for all $0<\vep<\vep_0$ there exists $K_{n}(\vep)>0$ such that for $\mu_\Phi$-a.e. point $x$ there exists
a sequence $n_k(x)\to \infty$ with
$$
K_{n_k(x)}(\vep)^{-1}
    \leq \frac{\mu_\Phi(B(x,n_{k}(x),\vep))}{e^{-n_{k}(x)P + S_{n_{k}(x)}\phi(x)}}
    \leq K_{n_k(x)}(\vep).
$$
Using $
\underline{X}_{J}\subset \bigcup_{\ell\ge 1} \bigcap_{j\ge \ell} X_{J_{\delta},j}
$
where $X_{J,n}=\{x \in M : \frac{1}{n}S_{n}\psi(x) \in J \}$ it is not difficulty check that for all $x \in \underline{X}_{J}$ there is a sequence of positive numbers $(m_j(x))_{j \in \N}$ converging to infinity such that
$x \in X_{J_{\frac{\delta}{2}},m_j(x)}$ and $m_j(x)$ is a moment where the Gibbs property holds.
Consider $\delta,\zeta>0$ arbitrary small,  $\alpha > P_{\topp}(f,\Phi) - L_{J_{\delta},\mu_\Phi} +\zeta$,
$\vep > 0$ small
and $N \in \N$ large. Take $m(x)\geq N$ so that $x\in X_{J_{\frac{\delta}{2}},m(x)}$, the constants satisfy
$K_{m(x)}(\vep) \leq e^{\frac{\zeta}{4}m(x)}$,
$K_{m(x)}(\frac{\vep}{2}) \leq e^{\frac{\zeta}{4}m(x)}$, and
\begin{equation*}
\mu_\Phi \Big(\{x \in M : \frac{1}{m(x)}\psi_{m(x)}(x) \in J_{\delta} \}\Big)
    \leq e^{-(L_{J_{\delta},\mu_\Phi}-\frac{\zeta}{2}) m(x)}.
\end{equation*}
Setting $\mathcal{G}_N:=\{(x,m(x)):x \in \underline{X}_{J}\}$ we
prove the result just follow with the same estimates
used in the proof of Theorem~\ref{thm:B} and obtain that $P_{\underline{X}_{J} }(f , \Phi) \leq P_{\topp}(f , \Phi) - L_{J_{\delta},\mu_\Phi}$ as claimed.
\end{remark}

\subsection{Proof of Corollary \ref{nonadtreg}}

By \cite{Barreira} and \cite{Mummert}, since $\Phi=0$ clearly satisfies the bounded distortion condition $\mu_{0}$
is a Gibbs measure. So Theorem \ref{thm:B} implies that $h_{\overline{X}_{J} }(f ) \leq h_{\topp}(f ) - L_{J_{\delta},\mu_{0}},$ for all $\delta > 0$ sufficiently small. By the large deviations estimates from \cite{Va13} and Theorem~\ref{legtrans} we have that
$$
h_{\overline{X}_{J} }(f ) \leq h_{\topp}(f) - \inf_{s \in
J_{\delta}}I_{f, 0,\Psi}(s)
$$
for all $\delta > 0$ small. The Legendre transform of $\Psi$ is continuous. Hence
$$
h_{\overline{X}_{J} }(f ) \leq h_{\topp}(f) - \inf_{s \in J} I_{f,
0,\Psi}(s)
$$
For the lower bound we proceed as follows, with an estimate similar to \cite[Theorem B]{BV14}.
It follows from Barreira and Doutor~\cite{BD09}  that if
$X(\alpha) \neq \emptyset$ then
$h_{X(\alpha)}(f ) = \sup_{\eta \in \mathcal{M}_{1}(f)}\{h_{\eta}(f) : \mathcal{F}_{\ast}(\Psi, \eta ) = \alpha\}$.
Thus, if  $\overline{X}_{J} \neq \emptyset$ and $\mathcal{F}_{\ast}(\Psi , \mu_{0}) \notin J$ then Theorem~\ref{legtrans} (item ii.) yields  that the infimum of $\ds\inf_{s \in J} I_{f,0,\Psi}(s)$ is realized at a boundary point $c_{\ast}$ of $J$.
Thus:
\begin{align*}
h_{\topp}(f ) - I_{f,0,\Psi}(c_{\ast})
    & =  h_{X(c_{\ast})} (f) \leq h_{X(J)}(f,)  \\
        & \leq h_{X(\overline{J})}(f) \leq h_{\underline{X}_{J}} (f)    \\
    & \leq  h_{\overline{X}_{J}} (f) \leq h_{\topp}(f) - I_{f,0 ,\Psi}(c_{\ast}).
\end{align*}
In particular, we prove we prove that for $J_c= \mathbb R \setminus (\mathcal{F}_{\ast}(\Psi , \mu_{0}) - c,
\mathcal{F}_{\ast}(\Psi , \mu_{0}) + c)$ we get $\overline{X_{J_c}}=\overline{X}_{\mu_{0}, \Psi, c}$ and so
$$
h_{\overline{X}_{\mu_{0}, \Psi, c}} (f)= h_{\topp} (f) - \min\{I_{f, 0, \Psi}\big(\mathcal{F}_{\ast}(\Psi , \mu_{0}) + c\big) \;,\;
I_{f, 0 , \Psi}\big(\mathcal{F}_{\ast}(\Psi , \mu_{0}) -c\big)\}
$$
whenever the set
$ \overline{X}_{\mu_{0}, \Psi, c}$ is not empty. So by Theorem~\ref{legtrans} we deduce that
the function $\R^+_0 \ni c\mapsto h_{\overline{X}_{\mu_{0},\Psi,c}}(f)$ is strictly decreasing
and concave in a neighborhood of zero.

\section{Examples and applications}\label{Examples}

In this section we provide some applications of the theory concerning the study of some classes
of non additive sequences of potentials related to either Lyapunov exponents or entropy.
{\color{black}As we already mentioned our results are also valid for subshifts of finite type, with exactly the same proof.}

\subsection{Linear cocycles}

Here we consider cocycles over subshifts of finite type as considered by Feng, Lau and K\"aenm\"aki~\cite{Feng1, Feng2}.  Let $\si:\Si \to \Si$ be the shift map on the space
$\Si=\{1,\dots, \ell\}^{\mathbb N}$ endowed with the distance $d(x,y)=2^{-n}$ where $x=(x_j)_j$,
$y=(y_j)_j$ and $n=\min \{j \ge 0 : x_j\neq y_j\}$.
Consider matrices $M_1, \dots, M_\ell \in \cM_{d\times d}(\mathbb C)$ such that
for every $n\ge 1$ there exists $i_1,\dots, i_n\in\{1,\dots, \ell\}$ so that the product matrix
$M_{i_1}\dots M_{i_n}\neq 0$. Then, the topological pressure function is well defined as
$
P(q)=\lim_{n\to\infty} \frac1n \log \sum_{\iota \in \Si_n} \|M_\iota\|^q
$
where $\Si_n=\{1,\dots, \ell\}^n$ and for any $\iota=(i_1, \dots, i_n) \in \Si_n$ one considers the matrix
$M_\iota = M_{i_n} \dots M_{i_2} M_{i_1}$. Moreover, for any $\si$-invariant probability measure $\mu$
define also the maximal Lyapunov exponent of $\mu$ by
\begin{equation*}\label{eq.LyapunovLau}
M_*(\mu)=\lim_{n\to\infty} \frac1n \sum_{\iota\in\Si_n} \mu([\iota]) \log \|M_\iota\|
\end{equation*}
and it holds that $P(q) = \sup\{ h_\mu(\si) + q \, M_*(\mu) : \mu
\in \cM_\si \}$. Notice that this is the variational principle for
the  potentials $\Phi=\{\varphi_n\}$ where
$\varphi_n(x)= q \log \|M_{\iota_n(x)}\|$ and for any $x\in \Si$ we set
${\iota_n(x)}\in\Si_n$ as the only symbol such that $x$ belongs to
the cylinder $[\iota_n(x)]$.
From \cite[Proposition~1.2]{Feng2}, if the set of matrices $\{M_1, \dots, M_d\}$ is irreducible over $\mathbb C^d$, (i.e. there is no non-trivial subspace $V\subset \mathbb C^d$ such that $M_i(V)\subset V$ for all $i=1,\dots, \ell$) there exists a unique equilibrium state $\mu_q$ for $\si$ with respect to
$\Phi$ and it is a Gibbs measure: there exists $C>0$ such that
\begin{equation*}\label{eq:GibbsLau}
\frac1C
    \leq \frac{\mu_q([\iota_n])}{e^{-n P(q)} \|M_{\iota_n} \|^q   }
    \leq C
\end{equation*}
for all $\iota_n\in \Si_n$ and $n\ge 1$.
Since the potentials $\varphi_n=\log \| M_{\iota_n(x)}\|$ are  constant in $n$-cylinders the family of potentials $\Phi$ clearly satisfies the bounded distortion condition.
It follows as a consequence of the large deviations bound in \cite{Va13} and Theorem~\ref{thm:B} that taking  $\Psi=\Phi$ with $q=1$ and $c>0$, the set
$$
\overline{X}_c=\{x\in \Sigma :  \limsup_{n\to\infty} \Big| \frac1n \log \|M_{\iota_n(x)}\| - M_*(\mu_\Phi) \Big|>c \}
$$
of points whose exponential growth of $M_{\iota_n(x)}$ is $c$-far away from the maximal Lyapunov exponent
$M_*(\mu_\Phi)$ for infinitely many values of $n$  has topological pressure strictly smaller than $\Ptop(f,\Phi)$.
Moreover, with respect to the maximal entropy measure $\mu_0$ Corollary~\ref{nonadtreg} yields that
the topological pressure function $c\mapsto h_{\overline{X}_c}(f) $ is strictly decreasing and concave
for small positive $c$.

\subsection{Non-conformal repellers}

The following class of local diffeomorphisms was introduced by Barreira and Gelfert~\cite{BG06}
in the study of multifractal analysis for Lyapunov exponents associated to non-conformal repellers.
Let $f:\mathbb{R}^2\rightarrow\mathbb{R}^2$ be
a $C^1$ local diffeomorphism, and let $J\subset \mathbb{R}^2$ be a
compact $f$-invariant set. Following \cite{BG06}, we say that $f$ satisfies the following {\em cone condition}
on $J$ if there exist a number $b\leq 1$ and for each $x\in J$ there is a one-dimensional subspace
$E(x)\subset T_x\mathbb{R}^2$ varying continuously with $x$ such that
$
Df(x)C_b(x)\subset \{0\}\cup\mathrm{int} \, C_b(fx)
$
where $C_b(x)=\{(u,v)\in E(x)\bigoplus E(x)^{\bot}:~||v||\leq
b||u||\}$.
It follows from \cite[Proposition~4]{BG06} that the latter condition implies that both  families of potentials
given by $\Psi_1=\{\log \sigma_1(Df^n(x))\}$ and $\Psi_2=\{\log \sigma_2(Df^n(x)) \}$ are almost additive,
where $\sigma_1(L)\geq \sigma_2(L)$ stands for the singular values of the linear transformation
$L:\mathbb{R}^2\rightarrow\mathbb{R}^2$, i.e., the eigenvalues of $(L^{*}L)^{1/2}$  with $L^*$
denoting the transpose of $L$.
Assume that $J$ is a locally maximal topological mixing repeller of $f$ such that:
\begin{itemize}
\item[(i)] $f$ satisfies the cone condition on $J$, and
\item[(ii)] $f$ has bounded distortion on $J$, i.e., there exists some $\delta>0$ such that
$$ \sup_{n\ge 1}\frac 1n \log \sup\Big\{ ||Df^n(y)(Df^n(z))^{-1}||:~x\in J~\text{and}~y,z\in B(x,n, \delta) \Big\}<\infty.$$
\end{itemize}
Then it follows from \cite[Theorem~9]{Barreira} that there exists a
unique equilibrium state $\mu_i$ for $(f,\Phi_i)$ which is a  
weak Gibbs measure 
with respect to the family of potentials $\Phi_i$, for $i=1,2$. Moreover,
from \cite[Example~4.6]{Va13},
for any $c>0$ the tail of the convergence to the largest or smallest Lyapunov exponent
(corresponding respectively to $j=1$ or $j=2$)
\begin{align*}
\mu_i
    \left(
    \left\{x\in M: \left|\frac 1 n \log \sigma_j(Df^n(x)) - \lim_{n\to\infty} \frac 1 n \int \log \sigma_j(Df^n(x)) d\mu_i
    \right|>c \right\}\right)
\end{align*}
decays exponentially fast as $n\to\infty$.
Moreover, it follows from Corollary~\ref{cor:LDP} that this exponential decay rate varies continuously with $c$.

One other consequence is that, although the irregular sets associated to $\Psi_j=\{\log \sigma_j(Df^n(x))\}$ have full topological pressure (using \cite{ZZC11} and the fact that $f$ has the specification property) the set of
irregular points whose time-$n$ Lyapunov exponents remain $c$-far away from the corresponding mean
have topological pressure strictly smaller than the topological pressure of the system.

\subsection{Entropy and Gibbs measures}

Let $\si:\Si \to \Si$ be the shift map on the space $\Si=\{1,\dots, \ell\}^{\mathbb N}$ endowed with the distance $d(x,y)=2^{-n}$ where $x=(x_j)_j$, $y=(y_j)_j$ and $n=\min \{j \ge 0 : x_j\neq y_j\}$. Set $\Si_n=\{1,\dots, \ell\}^n$
and for any $\iota=(i_1, \dots, i_n) \in \Si_n$ consider the $n$-cylinders
$[\iota]=\{x\in \Sigma : x_j=i_j, \; \forall 1 \le j \le n\}$.

Let $\Phi=\{\varphi_n\}$
be an almost additive sequence of potentials with the bounded distortion property and $\mu_\Phi$ be the
unique equilibrium state for $f$ with respect to $\Phi$ given by \cite{Barreira}. Fix $C>0$ so that
for every $x\in \Sigma$
$$
\varphi_{n}(x) + \varphi_{m}(f^n(x)) -  C
    \le \varphi_{m+n}(x)
    \le \varphi_{n}(x) + \varphi_{m}(f^n(x)) + C.
 $$
Since $\mu_\Phi$ is Gibbs there exists $P\in\mathbb R$ and $K>0$ so that
$$
\frac1K
    \leq \frac{\mu_\Phi([\iota_n(x)])}{e^{-P n +\varphi_n(x)} }
    \leq K
$$
for every $n\ge 1$ and every $x\in \Sigma$. In consequence, if $\psi_n(x)=\log \mu_\Phi([\iota_n(x)])$ then
\begin{align*}
\exp \psi_{m+n}(x) & =\mu([\iota_{m+n}(x)])
         \le K \; e^{-P (m+n) +\varphi_{m+n}(x)} \\
         & \le K \; e^C \; e^{-P n +\varphi_{n}(x)} \; e^{-P m +\varphi_{m}(f^n(x))} \\
         & \leq K^3 \; e^C \; \exp \psi_n(x) \; \exp \psi_m(f^n(x))
\end{align*}
for every $n\ge 1$ and $x\in \Sigma$. Thus,
$\psi_{m+n}(x) \le \psi_{n}(x) + \psi_{m}(f^n(x)) +\tilde C$
with $\tilde C=C+ 3\log K$. Since the lower bound is completely analogous we deduce that $\Psi=\{\psi_n\}$
is almost additive and satisfies the bounded distortion condition since $\psi_n$ is constant on $n$-cylinders.
In particular these satisfy the hypothesis of Theorem~B in ~\cite{Va13} to deduce exponential large deviations.
In fact it is a simple computation to prove that if $\mu_\Phi$ is a weak Gibbs measure then the corresponding sequence of functions $\Psi$ as above are asymptotically additive, but we shall not prove or
use this fact here.
By \cite{ZZC11} either the convergence is uniform or the irregular
set has full topological pressure. In this case, since this set is
contained in the set of points for which
$$
\limsup_{n\to\infty} \Big|-\frac1n \log \mu_\Phi([\iota_n(x)]) - h_{\mu_\Phi}(f) \Big|>0
$$
this has also full topological pressure. From our Theorem~\ref{thm:B}, for any $c>0$ the set
of points so that
$$
\limsup_{n\to\infty} \Big|-\frac1n \log \mu_\Phi([\iota_n(x)]) - h_{\mu_\Phi}(f) \Big|>c
$$
has topological pressure strictly smaller than $\Ptop(f,\Phi)$.

\vspace{.3cm}
\subsection*{Acknowledgements.}
\textcolor{black}{The authors are deeply grateful to the anonymous referees for their comments and suggestions that helped to improve the manuscript.}
P.V. was supported by a fellowship by CNPq-Brazil and is grateful to Faculdade de Ci\^encias da Universidade do Porto for the excelent research conditions. P.V. is also grateful to Prof. Kenhoo Lee and Prof. Manseob Lee for the hospitality during the conference "Dynamical Systems and Related Topics" held in Daejeon, Korea where part of this work was develloped.

\bibliographystyle{alpha}

\end{document}